\documentclass[journal]{IEEEtran}

\usepackage{cite}
\usepackage[citecolor=black]{hyperref}
\usepackage{lineno}
\usepackage[cmex10]{amsmath}
\usepackage{amscd,amsfonts,amsmath,amssymb,mathrsfs,pifont,stmaryrd,tipa}
\usepackage{array,cases,dsfont,graphicx,texdraw}
\usepackage{wrapfig}
\usepackage{graphics} 
\usepackage{epsfig} 
\usepackage{stfloats}
\usepackage{subfig}
\usepackage{hyperref}
\hypersetup{
	colorlinks=true,
	linkcolor=black
}
\usepackage{color}
\usepackage{xcolor}

\newtheorem{Remark}{Remark}
\newtheorem{Corollary}{Corollary}

\newenvironment{Proof}{\noindent{\em Proof:\/}}{\hfill $\Box$\par}
\newtheorem{Theorem}{Theorem}
\newtheorem{Lemma}{Lemma}

\newtheorem{Assumption}{Assumption}

\input{tcilatex}                                                          

\IEEEoverridecommandlockouts                              
\title{
	\small{This work has been submitted to the IEEE for possible publication. Copyright may be transferred without notice, after which this version may no longer be accessible.
\vspace{10mm}}\\
\LARGE \bf Distributed Robust Nash Equilibrium Seeking for Mixed-Order Games by a Neural-Network based Approach}

\author{Maojiao Ye, \emph{Member, IEEE}, Lei Ding, \emph{Senior Member, IEEE}, and Jizhao Yin
\thanks{M. Ye and J. Yin are with the School of Automation, Nanjing University of Science and Technology, 210094, P. R. China (Email: ye0003ao@e.ntu.edu.sg, yinjizhao@njust.edu.cn); L. Ding is with the Institute of Advanced Techology, Nanjing University of Posts and Telecommunications, 210023 (Email: dl522@163.com).}
\thanks{
	This work is supported by the National Natural Science Foundation of China (NSFC), No. 62222308, 62173181, 62073171,  the Natural Science Foundation of
	Jiangsu Province, No. BK20200744, BK20220139, Jiangsu Specially-Appointed Professor, No. RK043STP19001, the Young Elite Scientists Sponsorship Program by
	CAST, No. 2021QNRC001, 1311 Talent Plan of Nanjing University of Posts
	and Telecommunications, and the Fundamental Research Funds for the Central Universities, No. 30920032203. (\emph{Corresponding
	author: Lei Ding}).}
}
\begin{document}

\maketitle
\thispagestyle{empty}
\pagestyle{empty}

\begin{abstract}
	In practical applications, decision-makers with heterogeneous dynamics may be engaged in the same decision-making process. This motivates us to study distributed Nash equilibrium seeking for games in which players are mixed-order (first- and second-order) integrators influenced by unknown dynamics and external disturbances in this paper.
 To solve this problem, we employ an adaptive neural network to manage unknown dynamics and disturbances, based on which a distributed Nash equilibrium seeking algorithm is developed by further adapting concepts from gradient-based optimization and multi-agent consensus. By constructing appropriate Lyapunov functions, we analytically prove convergence of the reported method. Theoretical investigations suggest that players' actions would be steered to an arbitrarily small neighborhood of the Nash equilibrium, which is also testified by simulations.
\end{abstract}


\section{Introduction}
Game theory acts as an  effective technique for investigating interactive decision-making situations involving multiple participants. Typical examples that fall into the game theoretic framework include economic dispatch problems \cite{2}, charging coordination among electric vehicles \cite{wanTSG21}, energy consumption coordination in the smart grid \cite{8}, 
global optimization \cite{SmirnovSMCA2021},
 and formation control of multi-agent systems \cite{4}. The wide applications of games motivate many researchers to direct their energies to the development of Nash equilibrium seeking algorithms, leading to fruitful results in this field. For instance, games in which players are first-order integrators, second-order integrators, high-order integrators and linear-invariant dynamic ones were respectively investigated in \cite{5}-\cite{11}. 
Games in which players are described by Euler-Lagrange systems were investigated in \cite{12} and hybrid games, in which both  discrete-time players and continuous-time players are engaged, were addressed by the authors in \cite{13}.
It is worth mentioning that most of existing works focused on games in which players have homogeneous dynamics and related results on games with heterogeneous players are quite limited.
However, with distinct computation abilities, working environment and dynamics, decision-makers exhibit significant and remarkable heterogeneity in various perspectives.    Steered by the incentive to accommodate heterogeneity among different entities,
 heterogeneous multi-agent systems have attracted quite a few attention.
 For example, the authors in \cite{16}-\cite{17} and \cite{23} were concerned with formation control, output regulation as well as optimal coordination in linear multi-agent systems respectively, in which the constant matrices associated with the agents' dynamics are different from each other. Nonlinear systems with heterogeneous dimensions were considered in \cite{18}.  Moreover, second-order systems with time-varying gains and distinct inertia were explored in \cite{20}.
Among various kinds of heterogeneities, systems with both first- and second-order agents are of great interest since velocity-driven vehicles may work and collaborate with acceleration-driven ones \cite{21}.

The authors in \cite{22} dealt with consensus for a category of multi-agent systems where the engaged agents are described by first-order integrators as well as second-order integrators without using knowledge on the agents' velocity information. In addition, average consensus problems were addressed in \cite{21} under similar settings.
Consensus protocols were considered to be subject to bounded delays  for mixed-order systems in a discrete-time scenario in  \cite{14}.  However, few results on games in which players' dynamics are of different order have been reported, especially when both nonlinear dynamics and disturbances are involved. Therefore, this paper focuses on the establishment of distributed Nash equilibrium seeking algorithms for games with mixed-order participants. Moreover, as in many practical situations, e.g., physical hydraulic systems \cite{32}, air hybrid vehicles \cite{33} and marine surface vessels \cite{34}, external disturbances and un-modeled dynamics are inevitable due to complex working environment of engineering actuators and limited knowledge about the explicit system model, this paper further considers that players' dynamics are influenced by  unknown nonlinear dynamics and time-varying disturbances. Noticing that radial basis function neural network (RBFNN) has been proven to be effective for approximating unknown continuous mappings over a compact domain (see, e.g., \cite{31}-\cite{37}), this paper takes the benefits of RBFNN to establish robust Nash equilibrium seeking strategies for the considered mixed-order games. With some preliminary findings presented in \cite{YinICCA2020},
we give the core contributions and novelties of this manuscript as follows.
\begin{enumerate}
  \item This paper accommodates games with mixed-order integrator-type players who are suffering from both unknown nonlinear dynamics and time-varying disturbances. In comparison with the state of art,  the setting has rarely been explored. The presented exploration provides a unified viewpoint on how to simultaneously deal with first- and second-order players and offers convenience for the applications of games in mixed-order multi-agent systems.
  \item Un-modeled but Lipschitz nonlinear dynamics and disturbances are addressed through adapting an adaptive neural network. By compensating players' dynamics with the approximated value generated by the neural network, a distributed Nash equilibrium seeking strategy is developed for mixed-order games. This paper significantly improves its conference version \cite{YinICCA2020} by considering the disturbances and nonlinear dynamics that are unknown beforehand.
  \item The convergence property of the reported algorithm is analytically investigated  on the basis of Lyapunov stability analysis. The mathematical investigations show that the reported method is capable of steering players' actions and velocities respectively to be arbitrarily close to  the Nash equilibrium and zero.
  \end{enumerate}

We organize the remaining sections as below. Related preliminary knowledge is offered in Section \ref{npre} and Section \ref{p1_res} shows the problem in consideration. Method development and the corresponding analysis are provided in Section \ref{main_ref}. Numerical illustrations are offered in Section \ref{simulat}. Furthermore, concluding statements are illustrated in Section \ref{conc}.

\section{Preliminaries}\label{npre}

\textbf{Algebraic  graph theory}:
Graph $\mathcal{G}=(\mathbb{N},\mathcal{E})$ is given by a vertex set $\mathbb{N}=\{1,2,\cdots,N\}$, together with the associated edge set $\mathcal{E}\subseteq \mathbb{N} \times \mathbb{N}$. This paper considers that $\mathcal{G}$ is undirected in the sense that for any $(i,j)\in\mathcal{E}$, we can derive that $(j,i)\in\mathcal{E}$. Furthermore, the graph is connected provided that for every pair of distinct vertices, there exists a path. The adjacency matrix and Laplacian matrix of $\mathcal{G}$ are defined as $\mathcal{A}=[a_{ij}]$ and  $\mathcal{L}=\mathcal{D}-\mathcal{A},$ respectively, in which $a_{ij}=1$ if $(j,i)\in \mathcal{E}$, else, $a_{ij}=0$ ($a_{ii}=0$), $\mathcal{D}$ is a diagonal matrix with its $i$th diagonal entry being $\sum_{j=1}^Na_{ij}$  and the notation $B=[b_{ij}]$ illustrates a matrix whose $(i,j)$th entry is $b_{ij}$  \cite{5}.

\textbf{Radial basis function neural networks:}
A continuous function $l(z):\mathbb{R}^N\rightarrow \mathbb{R}^N$ can be estimated on a compact domain $z\in \Omega_z\subset \mathbb{R}^N$ by
\begin{equation}
l_{NN}(z)=W^TS(z),
\end{equation}
in which $W\in \mathbb{R}^{q\times N}$ is an adjustable weight matrix and $q$ is the neuron number. Moreover, $S(z)=[s_1(z),s_2(z),\cdots,s_q(z)]^T$ is the activation function given by
\begin{equation}
s_i(z)=\text{exp}\left[\frac{-(z-\mu_i)^T(z-\mu_i)}{\rho_i^2}\right],  i=1,2,\cdots,q,
\end{equation}
in which $\mu_i=[\mu_{i1},\mu_{i2},\cdots,\mu_{iN}]^T$ denotes the center of the receptive field, and $\rho_i$ denotes the width of the Gaussian function \cite{31}.

Then, for $z\in\Omega_z$ and any arbitrary small positive constant $\bar{\varepsilon}$, there exist a weight matrix $W^*\in\mathbb{R}^{q\times N}$ and a neural number $q$ so that
\begin{equation}
l(z)=W^{*T}S(z)+\varepsilon,
\end{equation}
in which $\varepsilon$ is the estimation error that satisfies $|\varepsilon|\leq \bar{\varepsilon}$ \cite{31}.


\begin{Lemma}\cite{31}\label{lemma2}
Assume that $V(t)\geq 0$ is a continuous function defined for $t\geq0$ and $V(0)$ is bounded. Then, if
\begin{equation}
\dot{V}(t)\leq-aV(t)+b,
\end{equation}
where $a>0,b>0$ are constants, we can obtain that
\begin{equation}
V(t)\leq V(0)e^{-at}+\frac{b}{a}(1-e^{-at}).
\end{equation}
\end{Lemma}
\begin{Lemma}\cite{30}\label{lemma3}
For any $\epsilon>0$ and $\eta\in \mathbb{R}$,
\begin{equation}
0\leq\lvert \eta \lvert-\eta \text{tanh}(\frac{\eta}{\epsilon})\leq \mathcal{K}\epsilon,
\end{equation}
where  $\mathcal{K}=\text{e}^{-(\mathcal{K}+1)}$.
\end{Lemma}



\section{Problem Description}\label{p1_res}

Consider a game containing $N$ players, in which the player set is given by $\mathbb{N}=\{1,2,\cdots,N\}$. Suppose that $n$ ($n\geq 1$ and $n< N$) of them are first-order integrators whose actions are steered by
\begin{equation}\label{first}
\dot{x}_i=u_i+g_i(\mathbf{x})+d_i(t),i\in \mathbb{N}_f,
\end{equation}
in which $x_i\in\mathbb{R}$, $u_i\in\mathbb{R}$, $g_i(\mathbf{x})\in\mathbb{R},$ $d_i(t)\in\mathbb{R}$, respectively represent for the action, the control signal to be designed, the unknown dynamics and the external, time-varying disturbance of player $i$. Moreover, $\mathbf{x}$ is a vector containing all players' actions, i.e., $\mathbf{x}=[x_1,x_2,\cdots,x_N]^T$ and $\mathbb{N}_f$ is the set of first-order players, i.e., $\mathbb{N}_f=\{1,2,\cdots,n\}.$ Furthermore, assume that the rest of players are second-order integrators whose actions evolve according to
\begin{equation}\label{second}
\begin{aligned}
\dot{x}_i&=v_i,\\
\dot{v}_i&=u_i+g_i(\mathbf{x})+d_i(t),\ \ i\in \mathbb{N}_s,
\end{aligned}
\end{equation}
in which $x_i\in\mathbb{R},$ $v_i\in\mathbb{R},$ $u_i\in\mathbb{R}$, $g_i(\mathbf{x})\in\mathbb{R}$ and $d_i(t)\in\mathbb{R},$ respectively denote the action, velocity, control signal, unknown dynamics and disturbance of player $i$. In addition, $\mathbb{N}_s$ is the set of second-order integrators, i.e., $\mathbb{N}_s=\{n+1,n+2,\cdots,N\}.$
Based on the above notations, it is clear that $\mathbb{N}=\mathbb{N}_f\bigcup \mathbb{N}_s.$ Associate each player $i,i\in\mathbb{N}$ with a cost function $f_i(\mathbf{x}),$ which can be alternatively denoted as $f_i(x_i,\mathbf{x}_{-i})$ by defining $\mathbf{x}_{-i}=[x_1,x_2,\cdots,x_{i-1},x_{i+1},\cdots,x_N]^T.$ The purpose of this manuscript is to construct  control signals $u_i,i\in\mathbb{N}$ so that  players' actions $\mathbf{x}$ can be steered to the Nash equilibrium $\mathbf{x}^*=(x_i^*,\mathbf{x}_{-i}^*)$, that satisfies
\begin{equation}
f_i(x_i^*,\mathbf{x}_{-i}^*)\leq f_i(x_i,\mathbf{x}_{-i}^*),
\end{equation}
for $x_i\in \mathbb{R},i\in\mathbb{N}$.

For notational simplicity, let $\nabla_if_i(\mathbf{x})=\frac{\partial f_i(\mathbf{x})}{\partial x_i}$ and $\nabla_{ij}^2f_i(\mathbf{x})=\frac{\partial^2 f_i(\mathbf{x})}{\partial x_i\partial x_j}.$
The mathematical development of this paper is based on the subsequent conditions.

\begin{Assumption}\label{Assu_1}
For each $i\in\mathbb{N},$ $f_i(\mathbf{x})$ is $\mathcal{C}^2$ and $\nabla_if_i(\mathbf{x})$ is globally Lipschitz  with constant $\bar{l}_i$ for $\mathbf{x}\in \mathbb{R}^N$.
\end{Assumption}

\begin{Assumption}\label{ass_2}
The players can exchange information through an undirected and connected graph $\mathcal{G}$.
\end{Assumption}

For notational convenience, let $\mathcal{A}_0=\text{diag}\{a_{ij}\}$ for $i,j\in\mathbb{N}$ denote a diagonal matrix with its diagonal elements successively being $a_{11},a_{12},\cdots,a_{1N},a_{21},\cdots,a_{NN}$. Moreover,  let $\otimes$ denote the Kronecker product. Then, under Assumption \ref{ass_2}, $-(\mathcal{L}\otimes \mathbf{I}_{N\times N}+\mathcal{A}_0)$, in which $\mathbf{I}_{N\times N}$ is an identity matrix of dimension $N\times N$, is Hurwitz.  Hence,
$\mathbf{P}(\mathcal{L}\otimes \mathbf{I}_{N\times N}+\mathcal{A}_0)+(\mathcal{L}\otimes \mathbf{I}_{N\times N}+\mathcal{A}_0)\mathbf{P}=\mathbf{Q}$ for some symmetric positive definite matrices $\mathbf{P}$ and $\mathbf{Q}$ of compatible dimensions \cite{5}.

\begin{Assumption}\label{Assu_2}
For $\mathbf{x},\mathbf{z}\in \mathbb{R}^N,$
\begin{equation}\label{cond_1}
(\mathbf{x}-\mathbf{z})^T(\mathcal{P}(\mathbf{x})-\mathcal{P}
(\mathbf{z}))\geq m||\mathbf{x}-\mathbf{z}||^2,
\end{equation}
where $m>0$  is a constant and $\mathcal{P}(\mathbf{x})=[\nabla_1f_1(\mathbf{x}),\nabla_2f_2(\mathbf{x}),\cdots,\nabla_Nf_N(\mathbf{x})]^T.$
\end{Assumption}

\begin{Assumption}\label{ass4}
For $\mathbf{x}\in \mathbb{R}^N,$ $\nabla_{ij}^2f_{i}(\mathbf{x})$ is bounded for $i\in\mathbb{N}_s,j\in \mathbb{N}$.
\end{Assumption}

\begin{Assumption}\label{ass5}
For each $i\in \mathbb{N},$ $g_i(\mathbf{x})$ is globally Lipschitz with constant $\eta_i$ and $d_i(t)$ is bounded.
\end{Assumption}

\begin{Remark}
Note that in \cite{26}, it is required that the un-modeled dynamics $g_{i}(\mathbf{x})$ is sufficiently smooth with its first two partial derivatives being bounded provided that $\mathbf{x}$ is bounded. Similarly, the disturbance $d_i(t)$ is supposed to be sufficiently smooth with $\dot{d}_i(t)$ and $\ddot{d}_i(t)$ being bounded in \cite{WangAT2020}\cite{26}. From Assumption \ref{ass5}, we see that these conditions are relaxed to some extent in this paper. Moreover, compared with internal model based approaches in \cite{ZhangTcyber20}\cite{RomanoTCNS20}, we do not assume disturbances to be of specific forms and different from \cite{HuangTSMCA} that considered quadratic games, this paper considers games with general costs. Besides, this paper considers mixed-order system dynamics while in the aforementioned works, players' dynamics are of the same order. The heterogeneity would further introduce some difficulties in the establishment and analytical study of the seeking algorithms.
\end{Remark}

\section{Main Results}\label{main_ref}

In this section, a distributed Nash equilibrium seeking strategy will be developed on the basis of  adaptive neural networks, consensus algorithms and gradient-based optimization algorithms. Moreover, the corresponding convergence analysis will be provided.

\subsection{Method Establishment}

To realize disturbance rejection in the considered game, the core idea of this paper is to adapt RBFNN (see, e.g., \cite{31} and many other references) to accommodate the unknown disturbances and dynamics. With RBFNN, the control input of player $i$ for $i\in\mathbb{N}_f$ is designed as
\begin{equation}\label{contro1}
u_i=-k_1(x_i-z_i)-\hat{W}_i^TS_i(\mathbf{y}_i)-\phi_i,
\end{equation}
in which $k_1$ is a positive constant, $z_i\in \mathbb{R}$ and $\hat{W}_i\in\mathbb{R}^{q_i\times 1}$ ($q_i$ is the number of neurons for player $i$) are adaptively updated according to
\begin{equation}\label{z_i_pr}
\dot{z}_i=-k_2\nabla_i f_i(\mathbf{y}_i),
\end{equation}
in which $k_2$ is a positive constant, $\nabla_i f_i(\mathbf{y}_i)=\nabla_i f_i(\mathbf{x})\left.\right|_{\mathbf{x}=\mathbf{y}_i}$
and
\begin{equation}\label{ad1}
\dot{\hat{W}}_i=\beta S_i(\mathbf{y}_i)(x_i-z_i),
\end{equation}
if ${\hat{W}_i}^T\hat{W}_i<W_{\text{max}}$
or alternatively, ${\hat{W}_i}^T\hat{W}_i=W_{\text{max}}$ and  $(x_i-z_i){\hat{W}_i}^TS_i(\mathbf{y}_i)<0$, where $\beta$ and $W_{\text{max}}$ are positive constants. In addition,
\begin{equation}\label{ad2}
\dot{\hat{W}}_i=\beta S_i(\mathbf{y}_i)(x_i-z_i)-\beta\frac{(x_i-z_i){\hat{W}_i}^TS_i(\mathbf{y}_i)}{{\hat{W}_i}^T\hat{W}_i}\hat{W}_i,
\end{equation}
if  ${\hat{W}_i}^T\hat{W}_i=W_{\text{max}}$ and $(x_i-z_i){\hat{W}_i}^TS_i(\mathbf{y}_i)\geq0$. Note that it is required that ${\hat{W}_i}^T(0)\hat{W}_i(0)\leq W_{\text{max}}$, which can be achieved by chosen the initial value of $\hat{W}_i$ to be zero.
Moreover, in \eqref{contro1},
\begin{equation}
\phi_i=\delta \text{tanh}\left(\frac{\mathcal{K}\delta(x_{i}-z_{i})}{\epsilon}\right),
\end{equation}
in which $\epsilon>0,\delta>0$ are constants. Furthermore, $\mathbf{y}_i\in\mathbb{R}^N$ and is defined as $\mathbf{y}_i=[y_{i1},y_{i2},\cdots,y_{iN}]^T$ where $y_{ij}$ is produced by
\begin{equation}\label{eqd}
\dot{y}_{ij}=-k_3 \left(\sum_{k=1}^Na_{ik}(y_{ij}-y_{kj})+a_{ij}(y_{ij}-\bar{x}_j)\right),
j\in\mathbb{N},
\end{equation}
where $k_3>0$ is a constant,  $\bar{x}_j=z_j$ for $j\in\mathbb{N}_f$ and $\bar{x}_j=x_j$ for $j\in\mathbb{N}_s$.
\begin{Remark}
The control input designed for first-order integrator-type players in \eqref{contro1} contains a regulation term $x_i-z_i$, which is employed to regulate $x_i$ to $z_i$. As the purpose of this paper is to drive $\mathbf{x}$ to $\mathbf{x}^*$, such a regulation term actually transfers the problem to drive $\mathbf{z}$, defined as $\mathbf{z}=[z_1,z_2,\cdots,z_N]^T$, to $\mathbf{x}^*,$ which is achieved by
\eqref{z_i_pr} and \eqref{eqd}.  In addition, $\hat{W}_i^TS_i(\mathbf{y}_i)$ and $\phi_i$ are designed based on RBFNN to address unknown dynamics and time-varying disturbances.
\end{Remark}

 By similar ideas, for second-order players, the control input of player $i$ for $i\in\mathbb{N}_s$ is designed as
 \begin{equation}\label{control2}
 u_i=-k_2k_4\nabla_i f_i(\mathbf{y}_i)-k_4v_i-\hat{W}_i^TS_i(\mathbf{y}_i)-\phi_i,
 \end{equation}
 where $k_4>0$ is a constant and
 $\hat{W}_i$ is updated according to
 \begin{equation}\label{ad3}
 \dot{\hat{W}}_i=\beta S_i(\mathbf{y}_i)(k_2\nabla_i f_i(\mathbf{y}_i)+v_i),
 \end{equation}
if ${\hat{W}_i}^T\hat{W}_i<W_{\text{max}}$ or alternatively ${\hat{W}_i}^T\hat{W}_i=W_{\text{max}}$ and $(k_2\nabla_i f_i(\mathbf{y}_i)+v_i){\hat{W}_i}^TS_i(\mathbf{y}_i)<0$.
Moreover, if  ${\hat{W}_i}^T\hat{W}_i=W_{\text{max}}$ and $(k_2\nabla_i f_i(\mathbf{y}_i)+v_i){\hat{W}_i}^TS_i(\mathbf{y}_i)\geq 0,$
\begin{equation}\label{ad4}
\begin{aligned}
\dot{\hat{W}}_i=&\beta S_i(\mathbf{y}_i)(k_2\nabla_i f_i(\mathbf{y}_i)+v_i)\\
&-\beta\frac{(k_2\nabla_i f_i(\mathbf{y}_i)+v_i){\hat{W}_i}^TS_i(\mathbf{y}_i)}{{\hat{W}_i}^T\hat{W}_i}\hat{W}_i,
\end{aligned}
\end{equation}
where ${\hat{W}_i}^T(0)\hat{W}_i(0)\leq W_{\text{max}}$.

Furthermore,
\begin{equation}
\phi_i=\delta \text{tanh}\left(\frac{\mathcal{K}\delta (k_2\nabla_i f_i(\mathbf{y}_i)+v_i)}{\epsilon}\right),
\end{equation}
and
\begin{equation}\label{eqre}
\dot{y}_{ij}=-k_3 \left(\sum_{k=1}^Na_{ik}(y_{ij}-y_{kj})+a_{ij}(y_{ij}-\bar{x}_j)\right),
j\in\mathbb{N}.
\end{equation}

\begin{Remark}
The control input design for second-order players in \eqref{control2} is similar to the  control design in \eqref{contro1}, where   $\hat{W}_i^TS_i(\mathbf{y}_i)$ and $\phi_i$ are included to accommodate unknown dynamics and disturbances. Different from \eqref{contro1}, stabilization of players' velocities $v_i$ is needed and achieved by the negative feedback of velocity $v_i$ in \eqref{control2}. In addition, it should be noted that multi-agent consensus components in  \eqref{eqd} and \eqref{eqre} are of the same format but $\bar{x}_i=x_i$ for second-order players, while $\bar{x}_j=z_j$ for first-order integrators.
\end{Remark}

Recalling the dynamics of first- and second-order integrator-type players in \eqref{first} and \eqref{second}, we get that
for first-order players,
\begin{equation}\label{eqww1}
\begin{aligned}
\dot{\mathbf{x}}_f=&-k_1(\mathbf{x}_f-\mathbf{z}_f)-[\hat{W}_i^TS_i(\mathbf{y}_i)]_{\mathbb{N}_f}\\
&-[\phi_i]_{\mathbb{N}_f}+[g_i(\mathbf{x})]_{\mathbb{N}_f}+[d_i(t)]_{\mathbb{N}_f},\\
\dot{\mathbf{z}}_f=&-k_2[\nabla_i f_i(\mathbf{y}_i)]_{\mathbb{N}_f},
\end{aligned}
\end{equation}
and for second-order players,
\begin{equation}\label{eqww2}
\begin{aligned}
\dot{\mathbf{x}}_s=&\mathbf{v}_s,\\
\dot{\mathbf{v}}_s=&-k_4\mathbf{v}_s-k_2k_4[\nabla_i f_i(\mathbf{y}_i)]_{\mathbb{N}_s}-[\hat{W}_i^TS_i(\mathbf{y}_i)]_{\mathbb{N}_s}\\
&-[\phi_i]_{\mathbb{N}_s}+[g_i(\mathbf{x})]_{\mathbb{N}_s}+[d_i(t)]_{\mathbb{N}_s},
\end{aligned}
\end{equation}
and  for $\mathbf{y}=[\mathbf{y}_1^T,\mathbf{y}_2^T,\cdots, \mathbf{y}_N^T]^T,$
\begin{equation}\label{eqww3}
\dot{\mathbf{y}}=-k_3(\mathcal{L}\otimes \mathbf{I}_{N\times N}+\mathcal{A}_0)(\mathbf{y}-\mathbf{1}_N\otimes \bar{\mathbf{x}}),
\end{equation}
where $\bar{\mathbf{x}}=[\bar{x}_1,\bar{x}_2,\cdots,\bar{x}_N]^T,$ $\mathbf{x}_f=[x_i]_{\mathbb{N}_f},$ $\mathbf{z}_f=[z_i]_{\mathbb{N}_f}$,
$\mathbf{x}_s=[x_i]_{\mathbb{N}_s},$ $\mathbf{v}_s=[v_i]_{\mathbb{N}_s}$ and the notation $[p_i]_{\mathbb{N}_f}([p_i]_{\mathbb{N}_s})$ defines the concatenated vector of $p_i$ for $i\in \mathbb{N}_f (i\in \mathbb{N}_s)$.

By similar analysis in \cite{31}, the subsequent result, which is needed in the convergence analysis of the proposed method, can be obtained.
\begin{Lemma}\label{lemma_add}
By the adaptive laws in \eqref{ad1}-\eqref{ad2} and \eqref{ad3}-\eqref{ad4},
\begin{equation}
\hat{W}_i^T(t)\hat{W}_i(t)\leq W_{\text{max}},\forall i\in\mathbb{N},
\end{equation}
for all $t\geq 0.$
\end{Lemma}
\begin{proof}
The analysis follows that of \cite{31} and the details are provided in Section \ref{lemma_add_proof} for the convenience of  readers.
\end{proof}

In the subsequent section, the analytical investigation on the proposed method will be presented.
\subsection{Convergence Analysis}

Before we continue to present the convergence results, the following supportive result is provided.
\begin{Lemma}\label{lemma5}
Under Assumptions \ref{Assu_1}-\ref{ass5}, there exists a positive constant $k_2^*$ so that for each $k_2>k_2^*$, there exist positive constants $k_1^*(k_2),k_3^*(k_2)$ so that for $k_1>k_1^*,k_3>k_3^*,$ there exists a positive constant $k_4^*(k_2,k_3)$ so that for $k_4>k_4^*,$ $\mathbf{x}(t),$ $\mathbf{z}_f(t)$, $\mathbf{v}_s(t)$  and $\mathbf{y}(t)$ generated by the proposed method in \eqref{eqww1}-\eqref{eqww3} stay bounded given that their initial values are bounded.
\end{Lemma}
\begin{proof}
Let $\bar{v}_i=k_2\nabla_i f_i(\mathbf{y}_i)+v_i$ for
$i\in\mathbb{N}_s$ and $\bar{\mathbf{v}}_s=[\bar{v}_i]_{\mathbb{N}_s}.$ Then,
\begin{equation}
\bar{\mathbf{v}}_s=\mathbf{v}_s+k_2[\nabla_i f_i(\mathbf{y}_i)]_{\mathbb{N}_s}.
\end{equation}

Therefore,
\begin{equation}
\begin{aligned}
\dot{\mathbf{x}}_s=&\bar{\mathbf{v}}_s-k_2[\nabla_i f_i(\mathbf{y}_i)]_{\mathbb{N}_s},\\
\dot{\bar{\mathbf{v}}}_s=&\dot{\mathbf{v}}_s+k_2H_1[\dot{\mathbf{y}}_i]_{\mathbb{N}_s}\\
=&-k_4\bar{\mathbf{v}}_s+k_2H_1[\dot{\mathbf{y}}_i]_{\mathbb{N}_s}-[\hat{W}_i^TS_i(\mathbf{y}_i)]_{\mathbb{N}_s}\\
&-[\phi_i]_{\mathbb{N}_s}+[g_i(\mathbf{x})]_{\mathbb{N}_s}+[d_i(t)]_{\mathbb{N}_s},
\end{aligned}
\end{equation}
where $H_1\in \mathbb{R}^{(N-n)\times N(N-n)}$ is a matrix whose $i$th row is  $[\mathbf{0}_{N(i-1)}^T, \nabla_{i1}^2f_{i}(\mathbf{y}_i),\cdots,\nabla_{iN}^2f_{i}(\mathbf{y}_i),\mathbf{0}_{N(N-n-i)}^T]$ for $i\in\{1,2,3,\cdots,N-n\}$.

To obtain the conclusion, define
$V=\sum_{i=1}^4 V_i,$ in which
\begin{equation}
\begin{aligned}
V_1=&\frac{1}{2}(\bar{\mathbf{x}}-\mathbf{x}^*)^T(\bar{\mathbf{x}}-\mathbf{x}^*),V_2=\frac{1}{2}\bar{\mathbf{v}}_s^T\bar{\mathbf{v}}_s,\\
V_3=&\frac{1}{2}(\mathbf{x}_f-\mathbf{z}_f)^T(\mathbf{x}_f-\mathbf{z}_f),\\
V_4=&(\mathbf{y}-\mathbf{1}_N\otimes \bar{\mathbf{x}})^T\mathbf{P} (\mathbf{y}-\mathbf{1}_N\otimes \bar{\mathbf{x}}).
\end{aligned}
\end{equation}
Then, by Assumption \ref{Assu_2},
\begin{equation}
\begin{aligned}
\dot{V}_1=&(\bar{\mathbf{x}}-\mathbf{x}^*)^T\dot{\bar{\mathbf{x}}}\\
=&(\bar{\mathbf{x}}-\mathbf{x}^*)[\dot{\mathbf{z}}_f^T,\dot{\mathbf{x}}_s^T]^T\\
=&-k_2(\bar{\mathbf{x}}-\mathbf{x}^*)^T[\nabla_i f_i(\mathbf{y}_i)]_{\mathbb{N}}\\
&+(\bar{\mathbf{x}}-\mathbf{x}^*)^T[\mathbf{0}_n^T,\bar{\mathbf{v}}_s^T]^T\\
=&-k_2(\bar{\mathbf{x}}-\mathbf{x}^*)^T(\mathcal{P}(\bar{\mathbf{x}})-\mathcal{P}(\mathbf{x}^*))\\
&+k_2(\bar{\mathbf{x}}-\mathbf{x}^*)^T(\mathcal{P}(\bar{\mathbf{x}})-[\nabla_i f_i(\mathbf{y}_i)]_{\mathbb{N}})\\
&+(\bar{\mathbf{x}}-\mathbf{x}^*)^T[\mathbf{0}_n^T,\bar{\mathbf{v}}_s^T]^T\\
\leq&-k_2m\Vert \bar{\mathbf{x}}-\mathbf{x}^*\Vert^2+\Vert \bar{\mathbf{x}}-\mathbf{x}^*\Vert\Vert \bar{\mathbf{v}}_s\Vert\\
&+k_2\max_{i\in\mathbb{N}}\{\bar{l}_i\}\Vert \bar{\mathbf{x}}-\mathbf{x}^*\Vert\Vert \mathbf{y}-\mathbf{1}_N\otimes \bar{\mathbf{x}}\Vert,
\end{aligned}
\end{equation}
where $\mathcal{P}(\bar{\mathbf{x}})=\mathcal{P}(\mathbf{x})|_{\mathbf{x}=\bar{\mathbf{x}}}$
and
\begin{equation}
\begin{aligned}
\dot{V}_2=&\bar{\mathbf{v}}_s^T(-k_4\bar{\mathbf{v}}_s+k_2H_1[\dot{\mathbf{y}}_i]_{\mathbb{N}_s})\\
&+\bar{\mathbf{v}}_s^T(-[\hat{W}_i^TS_i(\mathbf{y}_i)]_{\mathbb{N}_s}-[\phi_i]_{\mathbb{N}_s})\\
&+\bar{\mathbf{v}}_s^T([g_i(\mathbf{x})]_{\mathbb{N}_s}+[d_i(t)]_{\mathbb{N}_s})\\
=&-k_4\Vert \bar{\mathbf{v}}_s\Vert^2+k_2\bar{\mathbf{v}}_s^TH_1[\dot{\mathbf{y}}_i]_{\mathbb{N}_s}\\
&+\bar{\mathbf{v}}_s^T(-[\hat{W}_i^TS_i(\mathbf{y}_i)]_{\mathbb{N}_s}-[\phi_i]_{\mathbb{N}_s})\\
&+\bar{\mathbf{v}}_s^T([g_i(\mathbf{x})]_{\mathbb{N}_s}+[d_i(t)]_{\mathbb{N}_s})\\
\leq&-k_4\Vert \bar{\mathbf{v}}_s\Vert^2+a_s\Vert \bar{\mathbf{v}}_s\Vert+\bar{\mathbf{v}}_s^T([g_i(\mathbf{x})]_{\mathbb{N}_s}-[g_i(\mathbf{x}^*)]_{\mathbb{N}_s})\\
&+k_2k_3b\Vert \mathbf{y}-\mathbf{1}_N\otimes \bar{\mathbf{x}}\Vert\Vert \bar{\mathbf{v}}_s\Vert,
\end{aligned}
\end{equation}
in which $b=\text{sup}_{[\mathbf{y}_i]_{\mathbb{N}_s}\in \mathbb{R}^{(N-n)N}}\lVert H_1\lVert\lVert\mathcal{L}\otimes \mathbf{I}_{N\times N}+\mathcal{A}_0\lVert$,
$a_s=\sqrt{N-n}(\sqrt{q}\sqrt{W_{\text{max}}}+\delta+d+g)$, and $d,g$ are positive constants that satisfy $|d_i(t)|<d,|g_i(\mathbf{x}^*)|<g$ for $i\in\mathbb{N}$.

Moreover,
\begin{equation}
\begin{aligned}
\dot{V}_3=&(\mathbf{x}_f-\mathbf{z}_f)^T(\dot{\mathbf{x}}_f-\dot{\mathbf{z}}_f)\\
=&(\mathbf{x}_f-\mathbf{z}_f)^T(-k_1(\mathbf{x}_f-\mathbf{z}_f)+k_2[\nabla_i f_i(\mathbf{y}_i)]_{\mathbb{N}_f})\\
&+(\mathbf{x}_f-\mathbf{z}_f)^T(-[\hat{W}_i^TS_i(\mathbf{y}_i)]_{\mathbb{N}_f}-[\phi_i]_{\mathbb{N}_f})\\
&+(\mathbf{x}_f-\mathbf{z}_f)^T([g_i(\mathbf{x})]_{\mathbb{N}_f}+[d_i(t)]_{\mathbb{N}_f})\\
\leq&-k_1\Vert \mathbf{x}_f-\mathbf{z}_f\Vert^2+a_f\Vert \mathbf{x}_f-\mathbf{z}_f\Vert\\
&+k_2\sqrt{N}\max_{i\in\mathbb{N}}\{\bar{l}_i\}\Vert \mathbf{x}_f-\mathbf{z}_f\Vert\Vert \bar{\mathbf{x}}-\mathbf{x}^*\Vert\\
&+k_2\max_{i\in\mathbb{N}}\{\bar{l}_i\}\Vert \mathbf{x}_f-\mathbf{z}_f\Vert\Vert \mathbf{y}-\mathbf{1}_N\otimes \bar{\mathbf{x}}\Vert\\
&+(\mathbf{x}_f-\mathbf{z}_f)^T([g_i(\mathbf{x})]_{\mathbb{N}_f}-[g_i(\mathbf{x}^*)]_{\mathbb{N}_f}),
\end{aligned}
\end{equation}
where $a_f=\sqrt{n}(\sqrt{q}\sqrt{W_{\text{max}}}+\delta+d+g)$ and we have utilized that $\lVert[\nabla_i f_i(\mathbf{y}_i)]_{\mathbb{N}}\rVert=\lVert[\nabla_i f_i(\mathbf{y}_i)]_{\mathbb{N}}-\mathcal{P}(\bar{\mathbf{x}})+\mathcal{P}(\bar{\mathbf{x}})
-\mathcal{P}(\mathbf{x}^*) \rVert \leq \max_{i\in\mathbb{N}}\{\bar{l}_i\}\lVert \mathbf{y}-\mathbf{1}_N\otimes \bar{\mathbf{x}}\rVert+\sqrt{N}\max_{i\in\mathbb{N}}\{\bar{l}_i\}\lVert \bar{\mathbf{x}}-\mathbf{x}^*\rVert$ based on Assumption \ref{Assu_1}.

Furthermore,
\begin{equation}
\begin{aligned}
\dot{V}_4=&(\dot{\mathbf{y}}-\mathbf{1}_N\otimes \dot{\bar{\mathbf{x}}})^T\mathbf{P} (\mathbf{y}-\mathbf{1}_N\otimes \bar{\mathbf{x}})\\
&+(\mathbf{y}-\mathbf{1}_N\otimes \bar{\mathbf{x}})^T\mathbf{P} (\dot{\mathbf{y}}-\mathbf{1}_N\otimes \dot{\bar{\mathbf{x}}})\\
=&-k_3(\mathbf{y}-\mathbf{1}_N\otimes \bar{\mathbf{x}})^\text{T}(\mathcal{L}\otimes \mathbf{I}_{N\times N}+\mathcal{A}_0)\mathbf{P}(\mathbf{y}-\mathbf{1}_N\otimes \bar{\mathbf{x}})\\
&-k_3(\mathbf{y}-\mathbf{1}_N\otimes \bar{\mathbf{x}})^\text{T}\mathbf{P}(\mathcal{L}\otimes \mathbf{I}_{N\times N}+\mathcal{A}_0)(\mathbf{y}-\mathbf{1}_N\otimes \bar{\mathbf{x}})\\
&-2(\mathbf{y}-\mathbf{1}_N\otimes \mathbf{x})^\text{T}\mathbf{P}\mathbf{1}_N\otimes\dot{\bar{\mathbf{x}}}\\
\leq&-k_3\lambda_{min}(\mathbf{Q}){\lVert \mathbf{y}-\mathbf{1}_N \otimes \bar{\mathbf{x}} \rVert}^2\\
&+2k_2(\mathbf{y}-\mathbf{1}_N\otimes \bar{\mathbf{x}})^\text{T}\mathbf{P}\mathbf{1}_N\otimes [\nabla_i f_i(\mathbf{y}_i)]_{\mathbb{N}}\\
&-2(\mathbf{y}-\mathbf{1}_N\otimes \bar{\mathbf{x}})^\text{T}\mathbf{P}\mathbf{1}_N\otimes [\mathbf{0}_n^T,\bar{\mathbf{v}}_s^T]^T\\
\leq&-k_3\lambda_{min}(\mathbf{Q}){\lVert \mathbf{y}-\mathbf{1}_N \otimes \bar{\mathbf{x}} \rVert}^2\\
&+2k_2\sqrt{N}\max_{i\in\mathbb{N}}\{\bar{l}_i\}\lVert\mathbf{P }\rVert {\lVert \mathbf{y}-\mathbf{1}_N\otimes \bar{\mathbf{x}} \rVert}^2\\
&+2k_2N\max_{i\in\mathbb{N}}\{\bar{l}_i\}\lVert\mathbf{P }\rVert \lVert \mathbf{y}-\mathbf{1}_N\otimes \bar{\mathbf{x}} \rVert\lVert \bar{\mathbf{x}}-\mathbf{x}^* \rVert\\
&+2\sqrt{N}\lVert\mathbf{P }\rVert\lVert \mathbf{y}-\mathbf{1}_N\otimes \bar{\mathbf{x}} \rVert\lVert \bar{\mathbf{v}}_s \rVert,
\end{aligned}
\end{equation}
where the notation $\lambda_{min}(\mathbf{Q})$ denotes the minimum eigenvalue of $\mathbf{Q}$.
Hence,
\begin{equation}
\begin{aligned}
&\dot{V}\leq-k_2m\Vert \bar{\mathbf{x}}-\mathbf{x}^*\Vert^2-k_1\Vert \mathbf{x}_f-\mathbf{z}_f\Vert^2-k_4\Vert \bar{\mathbf{v}}_s\Vert^2\\
&-(k_3\lambda_{min}(\mathbf{Q})-2k_2\sqrt{N}\max_{i\in\mathbb{N}}\{\bar{l}_i\}\lVert\mathbf{P }\rVert){\lVert \mathbf{y}-\mathbf{1}_N \otimes \bar{\mathbf{x}} \rVert}^2\\
&+(1+\sqrt{N-n}\max_{i\in\mathbb{N}_s}\{\eta_i\})\Vert \bar{\mathbf{x}}-\mathbf{x}^*\Vert\Vert \bar{\mathbf{v}}_s\Vert+a_s\Vert \bar{\mathbf{v}}_s\Vert\\
&+(2k_2N\max_{i\in\mathbb{N}}\{\bar{l}_i\}\lVert\mathbf{P }\rVert+k_2\max_{i\in\mathbb{N}}\{\bar{l}_i\}) \\
&\times\lVert \mathbf{y}-\mathbf{1}_N\otimes \bar{\mathbf{x}} \rVert\lVert \bar{\mathbf{x}}-\mathbf{x}^* \rVert\\
&+(2\sqrt{N}\lVert\mathbf{P }\rVert+k_2k_3b)\lVert \mathbf{y}-\mathbf{1}_N\otimes \bar{\mathbf{x}} \rVert\lVert \bar{\mathbf{v}}_s \rVert\\
&+(k_2\sqrt{N}\max_{i\in\mathbb{N}}\{\bar{l}_i\}+\sqrt{n}\max_{i\in\mathbb{N}_f}\{\eta_i\})\Vert \mathbf{x}_f-\mathbf{z}_f\Vert\Vert \bar{\mathbf{x}}-\mathbf{x}^*\Vert\\
&+k_2\max_{i\in \mathbb{N}}\{\bar{l}_i\}\Vert \mathbf{x}_f-\mathbf{z}_f\Vert\Vert \mathbf{y}-\mathbf{1}_N\otimes \bar{\mathbf{x}}\Vert\\
&+a_f\Vert \mathbf{x}_f-\mathbf{z}_f\Vert+\sqrt{N-n}\max_{i\in\mathbb{N}_s}\{\eta_i\}||\bar{\mathbf{v}}_s||||\mathbf{x}_f-\mathbf{z}_f||\\
&+\sqrt{n}\max_{i\in\mathbb{N}_f}\{\eta_i\} ||\mathbf{x}_f-\mathbf{z}_f||^2.
\end{aligned}
\end{equation}

Therefore,
\begin{equation}
\begin{aligned}
\dot{V}\leq &-\bar{\Psi}_1||\bar{\mathbf{x}}-\mathbf{x}^*||^2-\bar{\Psi}_2||\mathbf{x}_f-\mathbf{z}_f||^2-\bar{\Psi}_3||\bar{\mathbf{v}}_s||^2\\
&-\bar{\Psi}_4||\mathbf{y}-\mathbf{1}_N\otimes \bar{\mathbf{x}}||^2+a_f||\mathbf{x}_f-\mathbf{z}_f||+a_s||\bar{\mathbf{v}}_s||,
\end{aligned}
\end{equation}
where $\bar{\Psi}_1=k_2m-\frac{1+\sqrt{N-n}\max_{i\in\mathbb{N}_s}\{\eta_i\}}{2}-1,$ $\bar{\Psi}_2=k_1-\sqrt{n}\max_{i\in\mathbb{N}_f}\{\eta_i\}-\frac{(k_2\sqrt{N}\max_{i\in\mathbb{N}}\{\bar{l}_i\}+\sqrt{n}\max_{i\in\mathbb{N}_f}\{\eta_i\})^2}{2}-\frac{(k_2\max_{i\in \mathbb{N}}\{\bar{l}_i\})^2}{2}-\frac{\sqrt{N-n}\max_{i\in\mathbb{N}_s}\{\eta_i\}}{2}$, $\bar{\Psi}_3=k_4-\frac{1+\sqrt{N-n}\max_{i\in\mathbb{N}_s}\{\eta_i\}}{2}-\frac{(2\sqrt{N}\lVert\mathbf{P }\rVert+k_2k_3b)^2}{2}-\frac{\sqrt{N-n}\max_{i\in\mathbb{N}_s}\{\eta_i\}}{2}$ and $\bar{\Psi}_4=k_3\lambda_{min}(\mathbf{Q})-2k_2\sqrt{N}\max_{i\in\mathbb{N}}\{\bar{l}_i\}\lVert\mathbf{P }\rVert-1-\frac{(2k_2N\max_{i\in\mathbb{N}}\{\bar{l}_i\}\lVert\mathbf{P }\rVert+k_2\max_{i\in\mathbb{N}}\{\bar{l}_i\})^2}{2}$.

Hence, by choosing $k_2$ to be sufficiently large, $\bar{\Psi}_1>0$. Then, for fixed $k_2$, we can choose $k_1$ and $k_3$ to be sufficiently large such that $\bar{\Psi}_2>0$ and $\bar{\Psi}_4>0$. Then, for fixed $k_1,k_2,k_3$, we can choose $k_4$ to be sufficiently large such that $\bar{\Psi}_3>0$. By such a tuning rule,
\begin{equation}
\dot{V}\leq-\frac{\min\{\bar{\Psi}_1,\bar{\Psi}_2,\bar{\Psi}_3,\bar{\Psi}_4\}}{\max\{\lambda_{max}(\mathbf{P}),\frac{1}{2}\}}V+a_f||\mathbf{x}_f-\mathbf{z}_f||+a_s||\bar{\mathbf{v}}_s||,
\end{equation}
i.e.,
\begin{equation}
\dot{V}\leq-\frac{\min\{\bar{\Psi}_1,\bar{\Psi}_2,\bar{\Psi}_3,\bar{\Psi}_4\}}{2\max\{\lambda_{max}(\mathbf{P}),\frac{1}{2}\}}V,
\end{equation}
for $\sqrt{V}\geq \frac{2(a_f+a_s)\max\{\lambda_{max}(\mathbf{P}),\frac{1}{2}\}}{\min\{\bar{\Psi}_1,\bar{\Psi}_2,\bar{\Psi}_3,\bar{\Psi}_4\}}$
from which the conclusion can be easily derived.

\end{proof}

From Lemma \ref{lemma5}, it can be concluded that $\mathbf{y}_i$ for $i\in\mathbb{N}$ would stay bounded given that the control gains are suitably chosen and the initial values of the variables are bounded. If this is the case, it is clear that for any positive constant $\bar{\varepsilon},$ there exist $W_i^*$ and $q_i$ that satisfy
\begin{equation}\label{eqww4}
g_i(\mathbf{y}_i)={W_i^*}^TS_i(\mathbf{y}_i)+\varepsilon_i,
  \end{equation}
where $\varepsilon_i<\bar{\varepsilon}$ as $\mathbf{y}_i$ belongs to a compact set.

Therefore, by \eqref{eqww1}-\eqref{eqww3} and \eqref{eqww4}, it is derived that for first-order integrators,
\begin{equation}\label{eqww11}
\begin{aligned}
\dot{\mathbf{x}}_f=&-k_1(\mathbf{x}_f-\mathbf{z}_f)-[\tilde{W}_i^TS_i(\mathbf{y}_i)]_{\mathbb{N}_f}\\
&-[\phi_i]_{\mathbb{N}_f}+[g_i(\mathbf{x})-g_i(\mathbf{y}_i)]_{\mathbb{N}_f}+[d_i(t)+\varepsilon_i]_{\mathbb{N}_f},\\
\dot{\mathbf{z}}_f=&-k_2[\nabla_i f_i(\mathbf{y}_i)]_{\mathbb{N}_f},
\end{aligned}
\end{equation}
and for second-order players,
\begin{equation}\label{eqww12}
\begin{aligned}
\dot{\mathbf{x}}_s=&\mathbf{v}_s,\\
\dot{\mathbf{v}}_s=&-k_4\mathbf{v}_s-k_2k_4[\nabla_i f_i(\mathbf{y}_i)]_{\mathbb{N}_s}-[\tilde{W}_i^TS_i(\mathbf{y}_i)]_{\mathbb{N}_s}\\
&-[\phi_i]_{\mathbb{N}_s}+[g_i(\mathbf{x})-g_i(\mathbf{y}_i)]_{\mathbb{N}_s}+[d_i(t)+\varepsilon_i]_{\mathbb{N}_s}.
\end{aligned}
\end{equation}
In addition,
\begin{equation}\label{eqww13}
\dot{\mathbf{y}}=-k_3(\mathcal{L}\otimes \mathbf{I}_{N\times N}+\mathcal{A}_0)(\mathbf{y}-\mathbf{1}_N\otimes \bar{\mathbf{x}}).
\end{equation}

The subsequent supportive lemmas are given before we provide the convergence results.
\begin{Lemma}\label{lemma44}
Suppose that $W_i^{*T}W_i^*
\leq W_{\text{max}}$ for $i\in\mathbb{N}$. Then, for $i\in\mathbb{N}_f,$
\begin{equation}
{\tilde{W}_i}^T\left(\frac{\dot{\hat{W}}_i}{\beta}-S_i(\mathbf{y}_i)(x_i-z_i)\right)\leq 0,
\end{equation}
and  for  $i\in\mathbb{N}_s,$
\begin{equation}
{\tilde{W}_i}^T\left(\frac{\dot{\hat{W}}_i}{\beta}-S_i(\mathbf{y}_i)\bar{v}_i\right)\leq0,
\end{equation}
in which   $\tilde{W}_i=\hat{W}_i-W^*_i$ for $i\in\mathbb{N}$.
\end{Lemma}
\begin{proof}
See Section \ref{lemma44_proof}.
\end{proof}

\begin{Lemma}\label{lemma45}
Let $\delta \geq |\varepsilon_i|+|d_i(t)|$ for $i\in\mathbb{N}$ and $t\geq 0$. Then, for each $i\in\mathbb{N}_f,$
\begin{equation}
(x_i-z_i)(d_i(t)+\varepsilon_i-\phi_i)\leq \epsilon,\label{bds1}
\end{equation}
and for each $i\in\mathbb{N}_s,$
\begin{equation}
\bar{v}_i(d_i(t)+\varepsilon_i-\phi_i)\leq \epsilon. \label{bds2}
\end{equation}
\end{Lemma}
\begin{proof}
See Section \ref{lemma45_proof}.
\end{proof}

We are now well prepared to provide the convergence analysis for the system in \eqref{eqww11}-\eqref{eqww13}.

\begin{Theorem}\label{main}
Assume that Assumptions \ref{Assu_1}-\ref{ass5} hold and $\delta \geq |\varepsilon_i|+|d_i(t)|$, $W_i^{*T}W_i^*\leq W_{\text{max}}$ for $i\in\mathbb{N}$, $t\geq 0$. Then, for any pair of positive constants $\Lambda$ and $\Xi$, there exist positive constants $\beta^*$ and $k_2^*$ so that for $\beta>\beta^*$ and $k_2
>k_2^*$, there exist positive constants $k_1^*$ and $k_3^*$ so that for $k_1>k_1^*(k_2),k_3>k_3^*(k_2)$, there exists a positive constant $k_4^*(k_2,k_3)$  so that for $k_4>k_4^*,$
\begin{equation}
\Vert\mathbf{x}(t)-\mathbf{x}^*\Vert\leq\Xi,\forall t>T,
\end{equation}
for some $T\geq 0$ given that $\Vert [(\bar{\mathbf{x}}(0)-\mathbf{x}^*)^T, \mathbf{v}_s^T(0), (\mathbf{y}(0)-\mathbf{1}_N\otimes \bar{\mathbf{x}}(0))^T, (\mathbf{x}_f(0)-\mathbf{z}_f(0))^T]^T\Vert+\sum_{i=1}^N {\tilde{W}_i(0)}^T\tilde{W}_i(0)\leq\Lambda$.
\end{Theorem}
\begin{Proof}
Let $V=\sum_{i=1}^5V_i,$
where
\begin{equation}
\begin{aligned}
V_1=&\frac{1}{2}( \bar{\mathbf{x}}-\mathbf{x}^*)^T( \bar{\mathbf{x}}-\mathbf{x}^*),V_2=\frac{1}{2}\mathbf{\bar{v}} ^T_s\mathbf{\bar{v}}_s,\\
V_3=&\frac{1}{2}(\mathbf{x}_f-\mathbf{z}_f)^T(\mathbf{x}_f-\mathbf{z}_f),\\
V_4=&(\mathbf{y}-\mathbf{1}_N\otimes \bar{\mathbf{x}})^T\mathbf{P} (\mathbf{y}-\mathbf{1}_N\otimes \bar{\mathbf{x}}),\\
V_5=&\frac{1}{2\beta}\sum_{i=1}^N {\tilde{W}_i}^T\tilde{W}_i.
\end{aligned}
\end{equation}
Then, following the analysis in Lemma \ref{lemma5}, we get that
\begin{equation}
\begin{aligned}
\dot{V}_1\leq&-k_2m\Vert \bar{\mathbf{x}}-\mathbf{x}^*\Vert^2+\Vert \bar{\mathbf{x}}-\mathbf{x}^*\Vert\Vert \bar{\mathbf{v}}_s\Vert\\
&+k_2\max_{i\in\mathbb{N}}\{\bar{l}_i\}\Vert \bar{\mathbf{x}}-\mathbf{x}^*\Vert\Vert \mathbf{y}-\mathbf{1}_N\otimes \bar{\mathbf{x}}\Vert,
\end{aligned}
\end{equation}
and
\begin{equation}
\begin{aligned}
\dot{V}_2=&\bar{\mathbf{v}}_s^T\dot{\bar{\mathbf{v}}}_s\\
=&\bar{\mathbf{v}}_s^T(-k_4\bar{\mathbf{v}}_s+k_2H_1[\dot{\mathbf{y}}_i]_{\mathbb{N}_s})\\
&+\bar{\mathbf{v}}_s^T(-[\tilde{W}_i^TS_i(\mathbf{y}_i)]_{\mathbb{N}_s}-[\phi_i]_{\mathbb{N}_s}+[d_i(t)+\varepsilon_i]_{\mathbb{N}_s})\\
&+\bar{\mathbf{v}}_s^T[g_i(\mathbf{x})-g_i(\mathbf{y}_i)]_{\mathbb{N}_s}\\
=&-k_4\Vert \bar{\mathbf{v}}_s\Vert^2+k_2\bar{\mathbf{v}}_s^TH_1[\dot{\mathbf{y}}_i]_{\mathbb{N}_s}\\
&+\bar{\mathbf{v}}_s^T(-[\tilde{W}_i^TS_i(\mathbf{y}_i)]_{\mathbb{N}_s}-[\phi_i]_{\mathbb{N}_s}+[d_i(t)+\varepsilon_i]_{\mathbb{N}_s})\\
&+\bar{\mathbf{v}}_s^T[g_i(\mathbf{x})-g_i(\mathbf{y}_i)]_{\mathbb{N}_s}\\
\leq&-k_4\Vert \bar{\mathbf{v}}_s\Vert^2+k_2k_3b\Vert \mathbf{y}-\mathbf{1}_N\otimes \bar{\mathbf{x}}\Vert\Vert \bar{\mathbf{v}}_s\Vert\\
&-\bar{\mathbf{v}}_s^T[\tilde{W}_i^TS_i(\mathbf{y}_i)]_{\mathbb{N}_s}+(N-n)\epsilon\\
&+\bar{\mathbf{v}}_s^T[g_i(\mathbf{x})-g_i(\mathbf{y}_i)]_{\mathbb{N}_s},
\end{aligned}
\end{equation}
where the result in Lemma \ref{lemma45} has been utilized.

Moreover,
\begin{equation}
\begin{aligned}
\dot{V}_3=&(\mathbf{x}_f-\mathbf{z}_f)^T(\dot{\mathbf{x}}_f-\dot{\mathbf{z}}_f)\\
=&(\mathbf{x}_f-\mathbf{z}_f)^T(-k_1(\mathbf{x}_f-\mathbf{z}_f)+k_2[\nabla_i f_i(\mathbf{y}_i)]_{\mathbf{N}_f})\\
&+(\mathbf{x}_f-\mathbf{z}_f)^T(-[\tilde{W}_i^TS_i(\mathbf{y}_i)]_{\mathbb{N}_f}-[\phi_i]_{\mathbb{N}_f})\\
&+(\mathbf{x}_f-\mathbf{z}_f)^T[d_i(t)+\varepsilon_i]_{\mathbb{N}_f})\\
&+(\mathbf{x}_f-\mathbf{z}_f)^T[g_i(\mathbf{x})-g_i(\mathbf{y}_i)]_{\mathbb{N}_f}\\
\leq&-k_1\Vert \mathbf{x}_f-\mathbf{z}_f\Vert^2+n\epsilon\\
&+k_2\sqrt{N}\max_{i\in\mathbb{N}}\{\bar{l}_i\}\Vert \mathbf{x}_f-\mathbf{z}_f\Vert\Vert \bar{\mathbf{x}}-\mathbf{x}^*\Vert\\
&+k_2\max_{i\in \mathbb{N}}\{\bar{l}_i\}\Vert \mathbf{x}_f-\mathbf{z}_f\Vert\Vert \mathbf{y}-\mathbf{1}_N\otimes \bar{\mathbf{x}}\Vert\\
&-(\mathbf{x}_f-\mathbf{z}_f)^T[\tilde{W}_i^TS_i(\mathbf{y}_i)]_{\mathbb{N}_f}\\
&+(\mathbf{x}_f-\mathbf{z}_f)^T[g_i(\mathbf{x})-g_i(\mathbf{y}_i)]_{\mathbb{N}_f},
\end{aligned}
\end{equation}
where the result in Lemma \ref{lemma45} has been utilized
and
\begin{equation}
\begin{aligned}
\dot{V}_4\leq&-k_3\lambda_{min}(\mathbf{Q}){\lVert \mathbf{y}-\mathbf{1}_N \otimes \bar{\mathbf{x}} \rVert}^2\\
&+2k_2\sqrt{N}\max_{i\in\mathbb{N}}\{\bar{l}_i\}\lVert\mathbf{P }\rVert {\lVert \mathbf{y}-\mathbf{1}_N\otimes \bar{\mathbf{x}} \rVert}^2\\
&+2k_2N\max_{i\in\mathbb{N}}\{\bar{l}_i\}\lVert\mathbf{P }\rVert \lVert \mathbf{y}-\mathbf{1}_N\otimes \bar{\mathbf{x}} \rVert\lVert \bar{\mathbf{x}}-\mathbf{x}^* \rVert\\
&+2\sqrt{N}\lVert\mathbf{P }\rVert\lVert \mathbf{y}-\mathbf{1}_N\otimes \bar{\mathbf{x}} \rVert\lVert \bar{\mathbf{v}}_s \rVert.\\
\end{aligned}
\end{equation}
Furthermore,
\begin{equation}
\begin{aligned}
\dot{V}_5=&\sum_{i=1}^N\frac{{\tilde{W}_i}^T\dot{\tilde{W}}_i}{\beta}\\
=&\sum_{i=1}^n\frac{{\tilde{W}_i}^T\dot{\hat{W}}_i}{\beta}+\sum_{i=n+1}^N\frac{{\tilde{W}_i}^T\dot{\hat{W}}_i}{\beta}.\\
\end{aligned}
\end{equation}
Hence,
\begin{equation}
\begin{aligned}
&\dot{V}\leq-k_2m\Vert \bar{\mathbf{x}}-\mathbf{x}^*\Vert^2-k_1\Vert \mathbf{x}_f-\mathbf{z}_f\Vert^2-k_4\Vert \bar{\mathbf{v}}_s\Vert^2\\
&-(k_3\lambda_{min}(\mathbf{Q})-2k_2\sqrt{N}\max_{i\in\mathbb{N}}\{\bar{l}_i\}\lVert\mathbf{P }\rVert){\lVert \mathbf{y}-\mathbf{1}_N \otimes \bar{\mathbf{x}} \rVert}^2\\
&+\Vert \bar{\mathbf{x}}-\mathbf{x}^*\Vert\Vert \bar{\mathbf{v}}_s\Vert+N\epsilon+(2k_2N\max_{i\in\mathbb{N}}\{\bar{l}_i\}\lVert\mathbf{P }\rVert\\
&+k_2\max_{i\in \mathbb{N}}\{\bar{l}_i\}) \lVert \mathbf{y}-\mathbf{1}_N\otimes \bar{\mathbf{x}} \rVert\lVert \bar{\mathbf{x}}-\mathbf{x}^* \rVert\\
&+(2\sqrt{N}\lVert\mathbf{P }\rVert+k_2k_3b)\lVert \mathbf{y}-\mathbf{1}_N\otimes \bar{\mathbf{x}} \rVert\lVert \bar{\mathbf{v}}_s \rVert\\
&+k_2\sqrt{N}\max_{i\in\mathbb{N}}\{\bar{l}_i\}\Vert \mathbf{x}_f-\mathbf{z}_f\Vert\Vert \bar{\mathbf{x}}-\mathbf{x}^*\Vert\\
&+k_2\max_{i\in \mathbb{N}}\{\bar{l}_i\}\Vert \mathbf{x}_f-\mathbf{z}_f\Vert\Vert \mathbf{y}-\mathbf{1}_N\otimes \bar{\mathbf{x}}\Vert\\
&+\sum_{i=1}^n{\tilde{W}_i}^T\left(\frac{\dot{\hat{W}}_i}{\beta}-S_i(\mathbf{y}_i)(x_i-z_i)\right)\\
&+\sum_{i=n+1}^N{\tilde{W}_i}^T\left(\frac{\dot{\hat{W}}_i}{\beta}-S_i(\mathbf{y}_i)\bar{v}_i\right)\\
&+(\mathbf{x}_f-\mathbf{z}_f)^T[g_i(\mathbf{x})-g_i(\mathbf{y}_i)]_{\mathbb{N}_f}\\
&+\bar{\mathbf{v}}_s^T[g_i(\mathbf{x})-g_i(\mathbf{y}_i)]_{\mathbb{N}_s}.
\end{aligned}
\end{equation}

As $||\tilde{W}_i||=||\hat{W}_i-W_i^*||\leq ||W_i^*||+||\hat{W}_i||\leq 2\sqrt{W_{\text{max}}}$, we get that $4NW_{\text{max}}-\sum_{i=1}^N\tilde{W}_i^T\tilde{W}_i\geq 0.$ Hence, by further  utilizing the results in Lemma \ref{lemma44},
\begin{equation}
\begin{aligned}
&\dot{V} \leq-k_2m\Vert \bar{\mathbf{x}}-\mathbf{x}^*\Vert^2-k_1\Vert \mathbf{x}_f-\mathbf{z}_f\Vert^2-k_4\Vert \bar{\mathbf{v}}_s\Vert^2\\
&-(k_3\lambda_{min}(\mathbf{Q})-2k_2\sqrt{N}\max_{i\in\mathbb{N}}\{\bar{l}_i\}\lVert\mathbf{P }\rVert){\lVert \mathbf{y}-\mathbf{1}_N \otimes \bar{\mathbf{x}} \rVert}^2\\
&+(2k_2N\max_{i\in\mathbb{N}}\{\bar{l}_i\}\lVert\mathbf{P }\rVert+k_2\max_{i\in\mathbb{N}}\{\bar{l}_i\}) \\
&\times \lVert \mathbf{y}-\mathbf{1}_N\otimes \bar{\mathbf{x}} \rVert\lVert \bar{\mathbf{x}}-\mathbf{x}^* \rVert\\
&+(2\sqrt{N}\lVert\mathbf{P }\rVert+k_2k_3b)\lVert \mathbf{y}-\mathbf{1}_N\otimes \bar{\mathbf{x}} \rVert\lVert \bar{\mathbf{v}}_s \rVert\\
&+k_2\sqrt{N}\max_{i\in\mathbb{N}}\{\bar{l}_i\}\Vert \mathbf{x}_f-\mathbf{z}_f\Vert\Vert \bar{\mathbf{x}}-\mathbf{x}^*\Vert+\Vert \bar{\mathbf{x}}-\mathbf{x}^*\Vert\Vert \bar{\mathbf{v}}_s\Vert\\
&+k_2\max_{i\in\mathbb{N}}\{\bar{l}_i\}\Vert \mathbf{x}_f-\mathbf{z}_f\Vert\Vert \mathbf{y}-\mathbf{1}_N\otimes \bar{\mathbf{x}}\Vert\\
&-\sum_{i=1}^N {\tilde{W}_i}^T\tilde{W}_i+N\epsilon+4NW_{\text{max}}\\
&+(\mathbf{x}_f-\mathbf{z}_f)^T[g_i(\mathbf{x})-g_i(\mathbf{y}_i)]_{\mathbb{N}_f}\\
&+\bar{\mathbf{v}}_s^T[g_i(\mathbf{x})-g_i(\mathbf{y}_i)]_{\mathbb{N}_s}.
\end{aligned}
\end{equation}

Noticing that
\begin{equation}
\Vert \bar{\mathbf{x}}-\mathbf{x}^*\Vert\Vert \bar{\mathbf{v}}_s\Vert\leq \frac{1}{2}||\bar{\mathbf{x}}-\mathbf{x}^*||^2+\frac{1}{2}||\bar{\mathbf{v}}_s||^2,
\end{equation}
and
\begin{equation}
\begin{aligned}
&k_2\max_{i\in \mathbb{N}}\{\bar{l}_i\}(2N||\mathbf{P}||+1)||\mathbf{y}-\mathbf{1}_N\otimes \bar{\mathbf{x}}||||\bar{\mathbf{x}}-\mathbf{x}^*||\\
&\leq \frac{(2N\max_{i\in\mathbb{N}}\{\bar{l}_i\}||\mathbf{P}||+\max_{i\in \mathbb{N}}\{\bar{l}_i\})^2}{2}||\bar{\mathbf{x}}-\mathbf{x}^*||^2\\
&+\frac{k_2^2}{2}||\mathbf{y}-\mathbf{1}_N\otimes \bar{\mathbf{x}}||^2.
\end{aligned}
\end{equation}
 In addition,
\begin{equation}
\begin{aligned}
&(2\sqrt{N}||\mathbf{P}||+k_2k_3b)||\mathbf{y}-\mathbf{1}_N\otimes \bar{\mathbf{x}}||||\bar{\mathbf{v}}_s||\\
\leq &\left(\sqrt{N}||\mathbf{P}||+\frac{k_3^2b}{2}\right)||\bar{\mathbf{v}}_s||^2 \\ &+\left(\sqrt{N}||\mathbf{P}||+\frac{k_2^2b}{2}\right)||\mathbf{y}-\mathbf{1}_N\otimes \bar{\mathbf{x}}||^2.
\end{aligned}
\end{equation}

Furthermore,
\begin{equation}
\begin{aligned}
&k_2\sqrt{N}\max_{i\in\mathbb{N}}\{\bar{l}_i\}||\mathbf{x}_f-\mathbf{z}_f||||\bar{\mathbf{x}}-\mathbf{x}^*||\\
\leq & \frac{\sqrt{N}\max_{i\in\mathbb{N}}\{\bar{l}_i\}}{2}||\bar{\mathbf{x}}-\mathbf{x}^*||^2\\
&+\frac{\sqrt{N}\max_{i\in\mathbb{N}}\{\bar{l}_i\}k_2^2}{2}||\mathbf{x}_f-\mathbf{z}_f||^2,
\end{aligned}
\end{equation}
and
\begin{equation}
\begin{aligned}
&k_2\sqrt{N}\max_{i\in\mathbb{N}}\{\bar{l}_i\}||\mathbf{x}_f-\mathbf{z}_f||||\mathbf{y}-\mathbf{1}_N\otimes\bar{\mathbf{x}}||\\
\leq &\frac{\sqrt{N}\max_{i\in\mathbb{N}}\{\bar{l}_i\}}{2}||\mathbf{y}-\mathbf{1}_N\otimes \bar{\mathbf{x}}||^2\\
&+\frac{\sqrt{N}\max_{i\in\mathbb{N}}\{\bar{l}_i\}k_2^2}{2}||\mathbf{x}_f-\mathbf{z}_f||^2.
\end{aligned}
\end{equation}

Hence,
\begin{equation}
\begin{aligned}
\dot{V}\leq &-\bar{\Phi}_1||\bar{\mathbf{x}}-\mathbf{x}^*||^2\\
&-(k_4-\frac{1}{2}-\sqrt{N}||\mathbf{P}||-\frac{k_3^2b}{2})||\bar{\mathbf{v}}_s||^2\\
&-\left(k_1-\frac{\sqrt{N}\max_{i\in\mathbb{N}}\{\bar{l}_i\}k_2^2}{2}\right.\\
&\left.-\frac{k_2^2\sqrt{N}\max_{i\in\mathbb{N}}\{\bar{l}_i\}}{2}\right)||\mathbf{x}_f-\mathbf{z}_f||^2\\
&-\left(k_3\lambda_{min}(\mathbf{Q})-2k_2\sqrt{N}\max_{i\in\mathbb{N}}\{\bar{l}_i\}||\mathbf{P}||-\sqrt{N}||\mathbf{P}||\right.\\
&\left.-\frac{k_2^2}{2}-\frac{k_2^2b}{2}-\frac{\sqrt{N}\max_{i\in\mathbb{N}}\{\bar{l}_i\}}{2}\right)||\mathbf{y}-\mathbf{1}_N\otimes \bar{\mathbf{x}}||^2\\
&-\sum_{i=1}^N {\tilde{W}_i}^T\tilde{W}_i+(\mathbf{x}_f-\mathbf{z}_f)^T[g_i(\mathbf{x})-g_i(\mathbf{y}_i)]_{\mathbb{N}_f}\\
&+\bar{\mathbf{v}}_s^T[g_i(\mathbf{x})-g_i(\mathbf{y}_i)]_{\mathbb{N}_s}+N\epsilon+4NW_{\text{max}}.
\end{aligned}
\end{equation}
where $\bar{\Phi}_1=k_2m-\frac{1}{2}-\frac{(2N\max_{i\in\mathbb{N}}\{\bar{l}_i\}||\mathbf{P}||+\max_{i\in \mathbb{N}}\{\bar{l}_i\})^2}{2}-\frac{\sqrt{N}\max_{i\in\mathbb{N}}\{\bar{l}_i\}}{2}.$
Furthermore,
\begin{equation}
\begin{aligned}
&(\mathbf{x}_f-\mathbf{z}_f)^T[g_i(\mathbf{x})-g_i(\mathbf{y}_i)]_{\mathbb{N}_f}\\
\leq &\sqrt{n}\max_{i\in \mathbb{N}_f}\{\eta_i\}||\mathbf{x}_f-\mathbf{z}_f||^2\\
&+\max_{i\in \mathbb{N}_f}\{\eta_i\}||\mathbf{x}_f-\mathbf{z}_f||||\mathbf{y}-\mathbf{1}_N\otimes \bar{\mathbf{x}}||,
\end{aligned}
\end{equation}
and similarly,
\begin{equation}
\begin{aligned}
&\bar{\mathbf{v}}_s^T[g_i(\mathbf{x})-g_i(\mathbf{y}_i)]_{\mathbb{N}_s}\\
\leq & \sqrt{N-n}\max_{i\in\mathbb{N}_s}\{\eta_i\}||\bar{\mathbf{v}}_s||||\mathbf{x}_f-\mathbf{z}_f||\\
&+\max_{i\in \mathbb{N}_s}\{\eta_i\} ||\bar{\mathbf{v}}_s||||\mathbf{y}-\mathbf{1}_N\otimes \bar{\mathbf{x}}||.
\end{aligned}
\end{equation}

Let $\bar{\Phi}_2=k_4-\frac{1}{2}-\sqrt{N}||\mathbf{P}||-\frac{k_3^2b}{2}-\frac{\sqrt{N-n}\max_{i\in\mathbb{N}_s}\{\eta_i\}}{2}-\frac{\max_{i\in \mathbb{N}_s}\{\eta_i\}}{2}$, $\bar{\Phi}_3=k_1-\frac{\sqrt{N}\max_{i\in\mathbb{N}}\{\bar{l}_i\}k_2^2}{2}-\frac{k_2^2\sqrt{N}\max_{i\in\mathbb{N}}\{\bar{l}_i\}}{2}-\sqrt{n}\max_{i\in \mathbb{N}_f}\{\eta_i\}-\frac{\max_{i\in \mathbb{N}_f}\{\eta_i\}}{2}-\frac{\sqrt{N-n}\max_{i\in\mathbb{N}_s}\{\eta_i\}}{2},$ and $\bar{\Phi}_4=k_3\lambda_{min}(\mathbf{Q})-2k_2\sqrt{N}\max_{i\in\mathbb{N}}\{\bar{l}_i\}||\mathbf{P}||-\frac{k_2^2}{2}-\sqrt{N}||\mathbf{P}||-\frac{k_2^2b}{2}-\frac{\sqrt{N}\max_{i\in\mathbb{N}}\{\bar{l}_i\}}{2}-\frac{\max_{i\in \mathbb{N}_f}\{\eta_i\}}{2}-\frac{\max_{i\in \mathbb{N}_s}\{\eta_i\}}{2},$ then,
\begin{equation}
\dot{V}\leq -KV+N\epsilon+4NW_{\text{max}},
\end{equation}
where $K=\min\{2\bar{\Phi}_1,2\bar{\Phi}_2,2\bar{\Phi}_3,\frac{\bar{\Phi}_4}{\lambda_{max}(\mathbf{P})},2\beta\}.$

Hence, by Lemma \ref{lemma2},
\begin{equation}
V(t)\leq V(0)e^{-Kt}+\frac{N\epsilon+4NW_{\text{max}}}{K},
\end{equation}
where $K$ can be arbitrarily large by the following tuning rule: choose $k_2$ to be large enough so that $\bar{\Phi}_1$ is sufficiently large. Then, for fixed $k_2,$ choose $k_1,k_3$ such that $\bar{\Phi}_3$ and $\bar{\Phi}_4$
are sufficiently large. Then, for fixed $k_3$, choose $k_4$ to be large enough so that $\bar{\Phi}_2$ is sufficiently large. If this is the case, $K$ is sufficiently large with sufficiently large $\beta$, indicating that $V(t)$ is decaying to be arbitrarily close to  zero. Recalling the definitions of the Lyapunov candidate function and $\bar{\mathbf{v}}_s$, the conclusion can be obtained.
\end{Proof}

\begin{Remark}
As $\Lambda$ and $\Xi$ can be any positive constants, Theorem \ref{main} indicates that for any bounded initialization, the reported method \eqref{eqww11}-\eqref{eqww13} can drive $\mathbf{x}(t)$ to an arbitrary small neighborhood of $\mathbf{x}^*$. The main content of this paper focuses on distributed Nash equilibrium seeking for games involving mixed-order players. Note that when all players are first-order integrators, i.e., $n=N$ (second-order integrators, i.e., $n=0$), Theorem \ref{main} illustrates that the method in \eqref{contro1}-\eqref{eqd} (\eqref{control2}-\eqref{eqre}) steers players' actions to  an arbitrarily small neighborhood of $\mathbf{x}^*$ as well. \textbf{Therefore, the presented analysis actually provides a unified viewpoint for the analysis of both first- and second-order players.}
\end{Remark}

If there exist no unknown nonlinear and disturbance modulations (i.e., $g_i(\mathbf{x})+d_i(t)$) in the players' dynamics, the corresponding estimation module can be removed from the proposed algorithm. If this is the case, we get that for first-order players,
\begin{equation}\label{eqww1tt}
\begin{aligned}
\dot{\mathbf{x}}_f=&-k_1(\mathbf{x}_f-\mathbf{z}_f),\\
\dot{\mathbf{z}}_f=&-k_2[\nabla_i f_i(\mathbf{y}_i)]_{\mathbb{N}_f},
\end{aligned}
\end{equation}
and for second-order players,
\begin{equation}\label{eqww2tt}
\begin{aligned}
\dot{\mathbf{x}}_s=&\mathbf{v}_s,\\
\dot{\mathbf{v}}_s=&-k_4\mathbf{v}_s-k_2k_4[\nabla_i f_i(\mathbf{y}_i)]_{\mathbb{N}_s},
\end{aligned}
\end{equation}
with
\begin{equation}\label{eqww3tt}
\dot{\mathbf{y}}=-k_3(\mathcal{L}\otimes \mathbf{I}_{N\times N}+\mathcal{A}_0)(\mathbf{y}-\mathbf{1}_N\otimes \bar{\mathbf{x}}),
\end{equation}
where the definitions for the variables and gains follow those in \eqref{contro1}-\eqref{eqre}. In this case, the subsequent result can be derived.

\begin{Theorem}\label{main1}
Under Assumptions \ref{Assu_1}-\ref{ass5}, there exists a positive constant $k_2^*$ so that for $k_2
>k_2^*$, there exist positive constants $k_1^*$ and $k_3^*$ so that for $k_1>k_1^*(k_2),k_3>k_3^*(k_2)$, there exists a positive constant $k_4^*(k_2,k_3)$  so that for $k_4>k_4^*,$ the Nash equilibrium $\mathbf{x}^*$ is globally exponentially stable with the strategy in \eqref{eqww1tt}-\eqref{eqww3tt}.
\end{Theorem}
\begin{proof}
Define
\begin{equation}
\begin{aligned}
V=&\frac{1}{2}( \bar{\mathbf{x}}-\mathbf{x}^*)^T( \bar{\mathbf{x}}-\mathbf{x}^*)+\frac{1}{2}\mathbf{\bar{v}} ^T_s\mathbf{\bar{v}}_s\\
&+\frac{1}{2}(\mathbf{x}_f-\mathbf{z}_f)^T(\mathbf{x}_f-\mathbf{z}_f)\\
&+(\mathbf{y}-\mathbf{1}_N\otimes \bar{\mathbf{x}})^T\mathbf{P} (\mathbf{y}-\mathbf{1}_N\otimes \bar{\mathbf{x}}).
\end{aligned}
\end{equation}

Then, following the proof of Theorem \ref{main},
\begin{equation}
\begin{aligned}
\dot{V}\leq&-k_2m\Vert \bar{\mathbf{x}}-\mathbf{x}^*\Vert^2+\Vert \bar{\mathbf{x}}-\mathbf{x}^*\Vert\Vert \bar{\mathbf{v}}_s\Vert\\
&+k_2\max_{i\in\mathbb{N}}\{\bar{l}_i\}\Vert \bar{\mathbf{x}}-\mathbf{x}^*\Vert\Vert \mathbf{y}-\mathbf{1}_N\otimes \bar{\mathbf{x}}\Vert\\
&-k_4\Vert \bar{\mathbf{v}}_s\Vert^2+k_2k_3b\Vert \mathbf{y}-\mathbf{1}_N\otimes \bar{\mathbf{x}}\Vert\Vert \bar{\mathbf{v}}_s\Vert\\
&-k_1\Vert \mathbf{x}_f-\mathbf{z}_f\Vert^2\\
&+k_2\sqrt{N}\max_{i\in\mathbb{N}}\{\bar{l}_i\}\Vert \mathbf{x}_f-\mathbf{z}_f\Vert\Vert \bar{\mathbf{x}}-\mathbf{x}^*\Vert\\
&+k_2\max_{i\in \mathbb{N}}\{\bar{l}_i\}\Vert \mathbf{x}_f-\mathbf{z}_f\Vert\Vert \mathbf{y}-\mathbf{1}_N\otimes \bar{\mathbf{x}}\Vert\\
&-k_3\lambda_{min}(\mathbf{Q}){\lVert \mathbf{y}-\mathbf{1}_N \otimes \bar{\mathbf{x}} \rVert}^2\\
&+2k_2\sqrt{N}\max_{i\in\mathbb{N}}\{\bar{l}_i\}\lVert\mathbf{P }\rVert {\lVert \mathbf{y}-\mathbf{1}_N\otimes \bar{\mathbf{x}} \rVert}^2\\
&+2k_2N\max_{i\in\mathbb{N}}\{\bar{l}_i\}\lVert\mathbf{P }\rVert \lVert \mathbf{y}-\mathbf{1}_N\otimes \bar{\mathbf{x}} \rVert\lVert \bar{\mathbf{x}}-\mathbf{x}^* \rVert\\
&+2\sqrt{N}\lVert\mathbf{P }\rVert\lVert \mathbf{y}-\mathbf{1}_N\otimes \bar{\mathbf{x}} \rVert\lVert \bar{\mathbf{v}}_s \rVert.
\end{aligned}
\end{equation}

Let $\bar{\Phi}_1=k_2m-\frac{5}{2},$ $\bar{\Phi}_2=k_4-\frac{1}{2}-\frac{(k_2k_3b)^2}{2}-\sqrt{N}||\mathbf{P}||,$ $\bar{\Phi}_3=k_3\lambda_{min}(\mathbf{Q})-\frac{(k_2\max_{i\in\mathbb{N}}\{\bar{l}_i\})^2}{2}-\frac{1}{2}-\frac{\max_{i\in\mathbb{N}}\{\bar{l}_i\}k_2}{2}-2k_2\sqrt{N}\max_{i\in\mathbb{N}}\{\bar{l}_i\}\lVert\mathbf{P }\rVert-(k_2N\max_{i\in\mathbb{N}}\{\bar{l}_i\}\lVert\mathbf{P }\rVert)^2-\sqrt{N}\lVert\mathbf{P }\rVert,$ and
$\bar{\Phi}_4=k_1-\frac{(k_2\max_{i\in\mathbb{N}}\{\bar{l}_i\})^2N}{2}-\frac{k_2\max_{i\in\mathbb{N}}\{\bar{l}_i\}}{2}$.

Then, by choosing $k_2>\frac{5}{2m},$ we get that $\bar{\Phi}_1>0$ and then, for fixed $k_2, $ we can choose $k_1$ and $k_3$   to be sufficiently large such that $\bar{\Phi}_3>0,\bar{\Phi}_4>0$.  Moreover, for fixed $k_2,k_3$, we can choose $k_4$ such that $\bar{\Phi}_2>0.$ By such a tuning rule, we get that
\begin{equation}
\dot{V}\leq -\min\{\bar{\Phi}_1,\bar{\Phi}_2,\bar{\Phi}_3,\bar{\Phi}_4\}||\mathbf{E}||^2,
\end{equation}
in which $\mathbf{E}=[( \bar{\mathbf{x}}-\mathbf{x}^*)^T,\mathbf{\bar{v}} ^T_s,(\mathbf{x}_f-\mathbf{z}_f)^T,(\mathbf{y}-\mathbf{1}_N\otimes \bar{\mathbf{x}})^T]^T$. Recalling the definition of $V$, the conclusion is drawn.
\end{proof}

Compared with Theorem  \ref{main}, it can be seen that Theorem \ref{main1} improves the semi-global results in Theorem  \ref{main} to global versions without unknown dynamics and disturbances. In addition, Assumption \ref{ass4} can be further relaxed in this case and the corresponding result is stated below.

 \begin{Corollary}\label{main3}
Assume that Assumptions \ref{Assu_1}-\ref{Assu_2} and \ref{ass5} hold and $\nabla_{ij}f_i(\mathbf{x})$ for $i\in\mathbb{N},j\in \mathbb{N}_s$ are bounded if $\mathbf{x}$ is bounded. Then, for any bounded initial condition, there exists a positive constant $k_2^*$ so that for $k_2
>k_2^*$, there exist positive constants $k_1^*$ and $k_3^*$  so that for $k_1>k_1^*(k_2),k_3>k_3^*(k_2)$, there exists a positive constant $k_4^*(k_2,k_3)$  so that for $k_4>k_4^*,$ $\mathbf{x}(t)$ exponentially converges to $\mathbf{x}^*$ under \eqref{eqww1tt}-\eqref{eqww3tt}.
\end{Corollary}

Compared with Theorem \ref{main1}, Corollary  \ref{main3} illustrates that if Assumption 4 is not satisfied, the corresponding result is degraded to a semi-global counterpart by supposing that the initial values of the variables are bounded.

\section{Numerical Verification}\label{simulat}
This section offers numerical verification of the reported methods by a connectivity control game involving $5$ vehicles concerned in  \cite{26}. In the game,  the cost function of vehicle $i$ is
\begin{equation}
f_i(\mathbf{x})=h_i(x_i)+l_i(\mathbf{x}),
\end{equation}
where $x_i=[x_{i1},x_{i2}]^T\in \mathbb{R}^2$ and
\begin{equation}
h_i(x_i)=x^T_im_{ii}x_i+x^T_im_i+i,
\end{equation}
in which $m_{ii}=\begin{bmatrix}
                                                       i & 0 \\
                                                       0 & i \\
                                                     \end{bmatrix},m_i=[i,i]^T.
$ Moreover,
$l_1(\mathbf{x})={\lVert x_1-x_2\lVert}^2$, $l_2(\mathbf{x})={\lVert x_2-x_3\lVert}^2,$ $l_3(\mathbf{x})={\lVert x_3-x_2\lVert}^2,$
$l_4(\mathbf{x})={\lVert x_4-x_2\lVert}^2+{\lVert x_4-x_5\lVert}^2$
and $l_5(\mathbf{x})={\lVert x_5-x_1\lVert}^2$. In the presented example, $x^*_{i}=[-\frac{1}{2},-\frac{1}{2}]^T$ for $i\in \{1,2,3,4,5\}$ \cite{26}. In the upcoming simulations, it is assumed that vehicles $1$-$3$ are first-order integrators and vehicles $4$-$5$ are second-order integrators.

To be more specific, for $i\in\{1,2,3\},$
\begin{equation}
\dot{x}_i=u_i+g_i(\mathbf{x})+d_i(t),
\end{equation}
in which
$g_1(\mathbf{x})+d_1(t)=[x_{21}+\text{sin}(t), x_{22}+\text{sin}(t)]^T$, $g_2(\mathbf{x})+d_2(t)=[x_{21}^2+x_{31}+2\text{sin}(2t),x_{22}+2\text{sin}(2t)]^T$, $g_3(\mathbf{x})+d_3(t)=[3x_{31}+3\text{sin}(3t), 3x_{32}+3\text{sin}(3t)]^T$. In addition, for $i\in\{4,5\},$
\begin{equation}
\begin{aligned}
\dot{x}_i=&v_i,\\
\dot{v}_i=&u_i+g_i(\mathbf{x})+d_i(t),
\end{aligned}
\end{equation}
where  $g_4(\mathbf{x})+d_4(t)=[4x_{41}+4\text{sin}(4t),4x_{42}+4\text{sin}(4t)]^T$ and $g_5(\mathbf{x})+d_5(t)=[5x_{51}+5\text{sin}(5t),5x_{52}+5\text{sin}(5t)]^T.$

\begin{figure}
  \centering
  \includegraphics[scale=0.25]{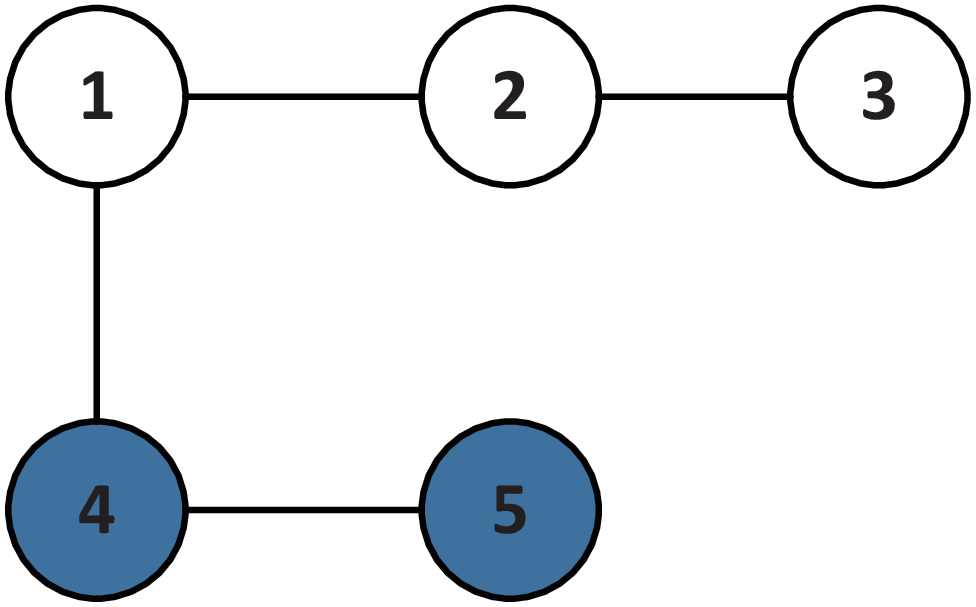}\\
  \caption{$\mathcal{G}$ among the vehicles.}\label{grad}
\end{figure}

In the simulation, the numbers of the neurons of the RBFNN are chosen as $11$ and the centers of RBFNN activation functions are $-2.5$, $-2$, $-1.5$, $-1$, $-0.5$, $0$, $0.5$, $1$, $1.5$, $2$, $2.5$ for all vehicles. Furthermore, the variances are all set as $5\sqrt{2}$. In addition $W_{\text{max}}=500$, $\beta=100$, $ \delta=10$, $\epsilon=0.01$ and $\hat{W}_i(0)$ is set as a zero matrix.

With $\mathbf{x}(0)=[-5,8,-4,-6,1,8,0,-8,-1,10]^T, \mathbf{v}_s(0)=[0,0,0,0]^T$, the numerical results produced by \eqref{eqww1}-\eqref{eqww3} are plotted in Figs. \ref{fig10}-\ref{fig12} by utilizing the communication graph in Fig. \ref{grad}. Fig. \ref{fig10}  plots players' actions from which it is clear that they would evolve to a small neighborhood of the Nash equilibrium. In addition, Fig. \ref{fig12} illustrates the evolution of  $\mathbf{v}_s(t),$ from which it can be seen that velocities of the second-order players would be driven to be sufficiently small. Hence, the result in Theorem \ref{main} is numerically testified.
\begin{figure}
  \centering
  \includegraphics[scale=0.5]{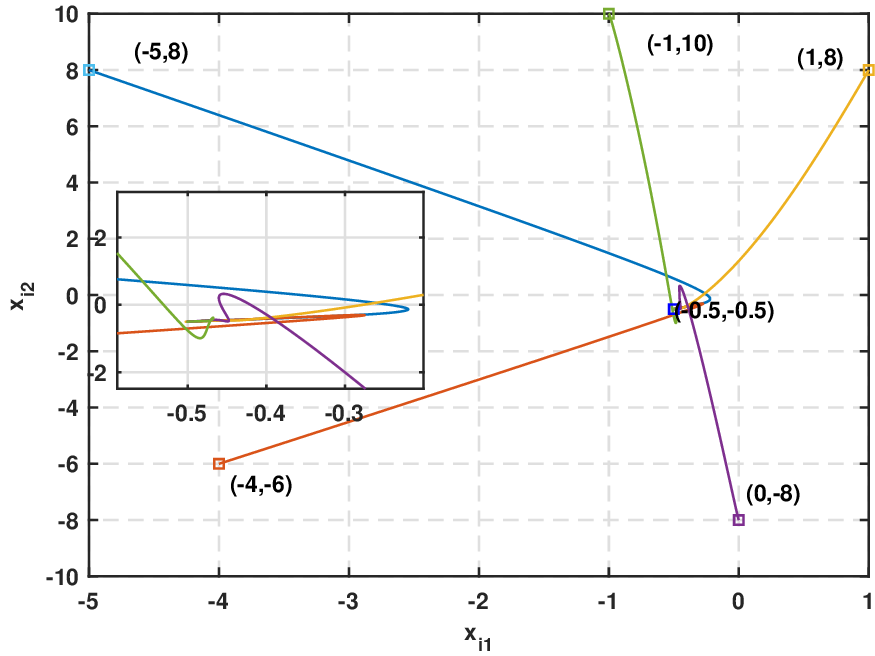}\\
  \caption{The evolutions of vehicles' positions generated by \eqref{eqww1}-\eqref{eqww3}. }\label{fig10}
\end{figure}
  \begin{figure}
  \centering
  \includegraphics[scale=0.5]{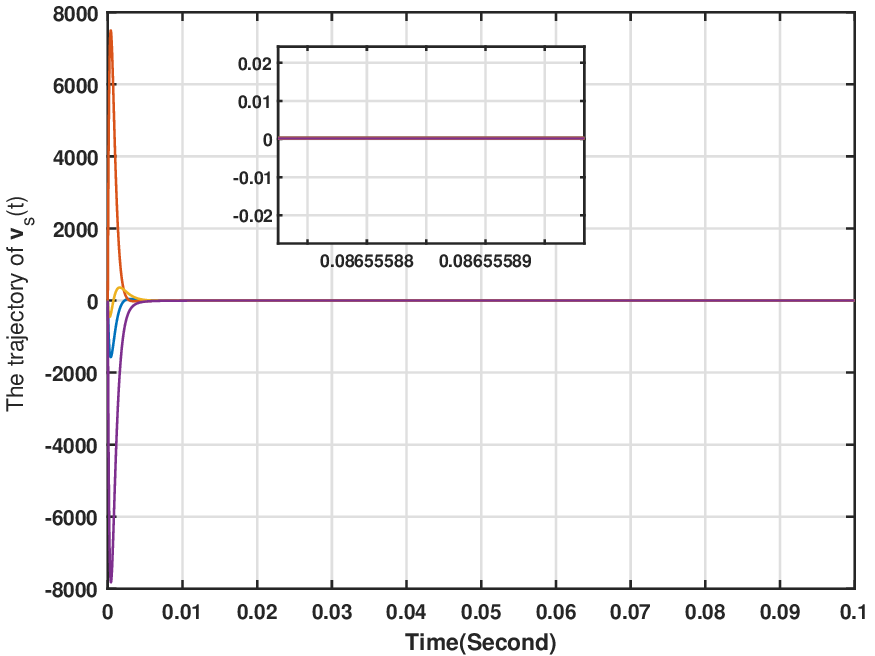}\\
  \caption{$\mathbf{v}_s(t)$ generated by \eqref{eqww1}-\eqref{eqww3}. }\label{fig12}
\end{figure}

Moreover, when there are no nonlinear dynamics and disturbances, the strategy in \eqref{eqww1tt}-\eqref{eqww3tt} is testified with the corresponding numerical results plotted in Figs. \ref{x}-\ref{v_sno_dis}.  Figs. \ref{x}-\ref{v_sno_dis} illustrate vehicles' positions and  velocities of the force-actuated vehicles, respectively. From these figures, it is seen that vehicles' positions evolve to be close to $\mathbf{x}^*$ and velocities of the second-order ones evolve to be close to zero. To this end, Theorem   \ref{main1} is testified.

  \begin{figure}
  \centering
  \includegraphics[scale=0.5]{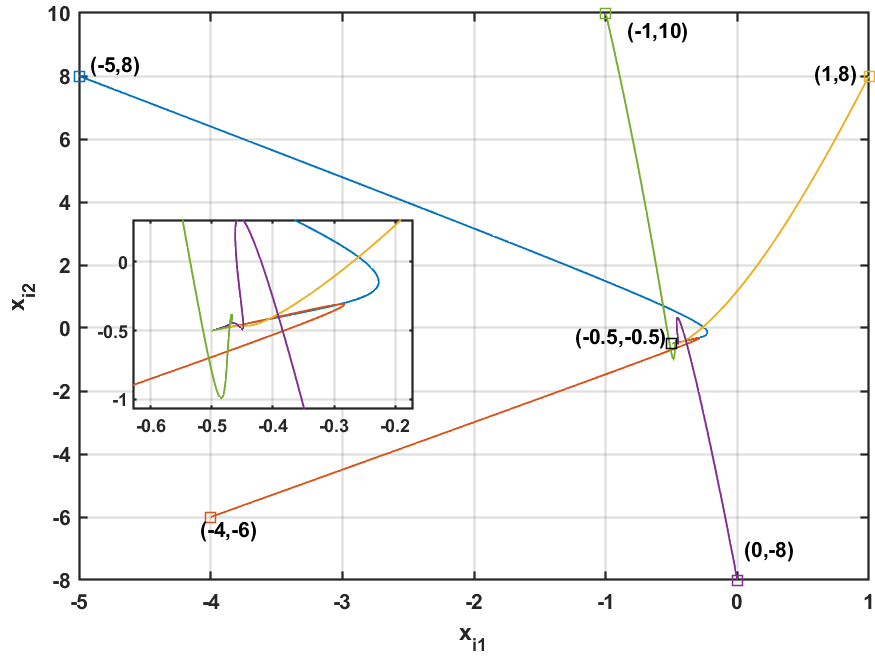}\\
  \caption{The evolutions of vehicles' positions produced by \eqref{eqww1tt}-\eqref{eqww3tt}. }\label{x}
\end{figure}

  \begin{figure}
  \centering
  \includegraphics[scale=0.5]{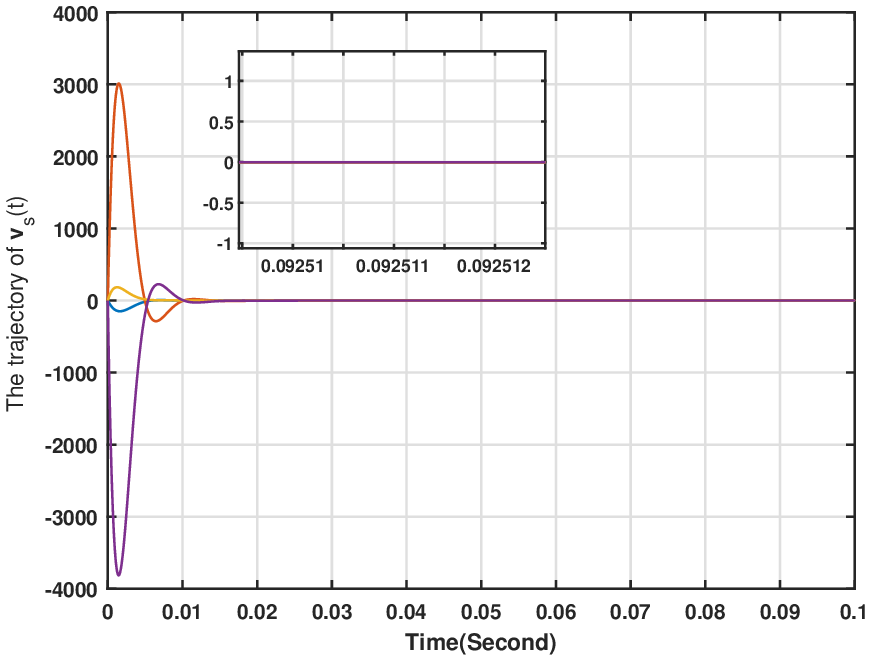}\\
  \caption{$\mathbf{v}_s(t)$ generated by \eqref{eqww1tt}-\eqref{eqww3tt}. }\label{v_sno_dis}
\end{figure}

\section{Conclusions}\label{conc}
This paper accommodates distributed Nash equilibrium seeking for mixed-order games with both first-order integrator-type participants and second-order integrator-type participants. In particular, players' dynamics are considered to be influenced by unknown but Lipschitz nonlinear dynamics and time-varying disturbances. To address unknown dynamics and achieve disturbance rejection, an adaptive neural network based approach, i.e., RBFNN, is adapted. Through suitably designing control inputs and choosing control parameters, it is proven that the reported methods are able to steer players' actions and velocities of second-order integrators to be arbitrarily close  to Nash equilibrium and zero, respectively.

\section{Appendix}
\subsection{Proof of Lemma \ref{lemma_add}}\label{lemma_add_proof}
For each $i=\mathbb{N}_f$, $\hat{W}_i$ is generated by \eqref{ad1}-\eqref{ad2}. For notational clarity, define $Y_{i}=\hat{W}_i^T\hat{W}_i$. Then, if $Y_{i}<W_{\text{max}},$ $\hat{W}_i^T(t)\hat{W}_i(t)\leq W_{\text{max}},\forall i\in\mathbb{N}$ holds. Moreover, if
$Y_{i}=W_{\text{max}}$ and $(x_i-z_i){\hat{W}_i}^TS_i(\mathbf{y}_i)<0$,
 \begin{equation*}
  \dot{Y}_{i}=2\hat{W}_i^T\dot{\hat{W}}_i=2\beta(x_i-z_i)\hat{W}_i^TS_i(\mathbf{y}_i)<0,
  \end{equation*}
  indicating that $Y_i$ deceases and hence $\hat{W}_i^T(t)\hat{W}_i(t)\leq W_{\text{max}},\forall i\in\mathbb{N}$ holds. In addition, if $Y_{i}=W_{\text{max}}$ and $(x_i-z_i){\hat{W}_i}^TS_i(\mathbf{y}_i)\geq0$,
   \begin{equation*}
  \begin{aligned}
  &\dot{Y}_{i}=2\hat{W}_i^T\dot{\hat{W}}_i\\
  &=2\beta(x_i-z_i)\hat{W}_i^TS_i(\mathbf{y}_i)-2\beta(x_i-z_i)\hat{W}_i^TS_i(\mathbf{y}_i)\\
  &=0,
  \end{aligned}
  \end{equation*}
  indicating that $\hat{W}_i^T(t)\hat{W}_i(t)\leq W_{\text{max}},\forall i\in\mathbb{N}$ holds.

  Summarizing the above cases, we get that for each $i\in\mathbb{N}_f,$ $\hat{W}_i^T(t)\hat{W}_i(t)\leq W_{\text{max}}$ holds.

  By similar arguments, it can be derived that for each $i\in\mathbb{N}_s,$  $\hat{W}_i^T(t)\hat{W}_i(t)\leq W_{\text{max}}$ holds as well, thus drawing the conclusion.
\subsection{Proof of Lemma \ref{lemma44}}\label{lemma44_proof}

For each $i\in\mathbb{N}_s$, if  $\dot{\hat{W}}_i=\beta S_i(\mathbf{y}_i)\bar{v}_i$
  \begin{equation*}
  {\tilde{W}_i}^T\left(\frac{\dot{\hat{W}}_i}{\beta}-S_i(\mathbf{y}_i)\bar{v}_i\right)=0.
  \end{equation*}
Moreover, if $\dot{\hat{W}}_i=\beta S_i(\mathbf{y}_i)\bar{v}_i-\beta\frac{\bar{v}_i{\hat{W}_i}^TS_i(\mathbf{y}_i)}{{\hat{W}_i}^T\hat{W}_i}\hat{W}_i$, we know that ${\hat{W}_i}^T\hat{W}_i=W_{\text{max}}$ and  $\bar{v}_i{\hat{W}_i}^TS_i(\mathbf{y}_i)\geq0$. If this is the case,
 \begin{equation*}
  {\tilde{W}_i}^T\left(\frac{\dot{\hat{W}}_i}{\beta}-S_i(\mathbf{y}_i)\bar{v}_i\right)=-\frac{\bar{v}_i{\hat{W}_i}^TS_i(\mathbf{y}_i)}{{\hat{W}_i}^T\hat{W}_i}(\tilde{W}_i^T\hat{W}_i),
  \end{equation*}
in which

\begin{equation*}
\begin{aligned}
\tilde{W}_i^T\hat{W}_i&=(\hat{W}_i-W^*_i)^T\hat{W}_i\\
&=\hat{W}_i^T\hat{W}_i-W_i^{*T}(\tilde{W}_i+W_i^*)\\
&=\hat{W}_i^T\hat{W}_i-W_i^{*T}W_i^{*}-(\hat{W}_i-\tilde{W}_i)^T\tilde{W}_i\\
&=\hat{W}_i^T\hat{W}_i-W_i^{*T}W_i^{*}+\tilde{W}_i^T\tilde{W}_i-\tilde{W}_i^T\hat{W}_i.
\end{aligned}
\end{equation*}

Hence, it can be obtained that
\begin{equation}
\tilde{W}_i^T\hat{W}_i=\frac{1}{2}(\hat{W}_i^T\hat{W}_i-W_i^{*T}W_i^{*}+\tilde{W}_i^T\tilde{W}_i)\geq0,
\end{equation}
in which we have utilized  the conclusions that $\hat{W}_i^T\hat{W}_i=W_{\text{max}}\geq W_i^{*T}W_i^{*}$ and $\tilde{W}_i^T\tilde{W}_i\geq0$.

Therefore,
\begin{equation}
{\tilde{W}_i}^T\left(\frac{\dot{\hat{W}}_i}{\beta}-S_i(\mathbf{y}_i)\bar{v}_i\right)\leq0
\end{equation}
for $i\in\mathbb{N}_s.$

Summarizing the above two cases, it can be obtained that
for  each $i\in\mathbb{N}_s,$
\begin{equation}
{\tilde{W}_i}^T\left(\frac{\dot{\hat{W}}_i}{\beta}-S_i(\mathbf{y}_i)\bar{v}_i\right)\leq0.
\end{equation}

By similar arguments, it can be easily obtained that
\begin{equation}
{\tilde{W}_i}^T\left(\frac{\dot{\hat{W}}_i}{\beta}-S_i(\mathbf{y}_i)(x_i-z_i)\right)\leq0,
\end{equation}
for $i\in\mathbb{N}_f$.
\subsection{Proof of Lemma \ref{lemma45}}\label{lemma45_proof}

For each $i\in\mathbb{N}_f,$ we have
\begin{equation}
\begin{aligned}
&(x_i-z_i)(d_i(t)+\varepsilon_i-\phi_i)\\
\leq&|x_i-z_i|| d_i(t)+\varepsilon_i |-(x_i-z_i)^T\phi_i\\
\leq&\delta| x_i-z_i|-\delta(x_i-z_i)\text{tanh}\left(\frac{\mathcal{K}\delta(x_i-z_i)}{\epsilon}\right)\\
\leq&\epsilon,
\end{aligned}
\end{equation}
by Lemma \ref{lemma3}. By similar arguments, \eqref{bds2} can be obtained.

\end{document}